\documentclass[a4paper,10pt]{amsart}
\usepackage{amsmath,amssymb}
\usepackage{amscd}
\newtheorem{thm}{Theorem}[section]
\newtheorem{prop}[thm]{Proposition}
\newtheorem{lemma}[thm]{Lemma}
\newtheorem{cor}[thm]{Corollary}

\theoremstyle{remark}
\newtheorem{remark}[thm]{Remark}
\newcommand{\id}{{\rm{id}}}
\newcommand{\Ad}{{\rm{Ad}}}

\newcommand{\BC}{\mathbf C}
\newcommand{\BZ}{\mathbf Z}

\newcommand{\BT}{\mathbf T}
\newcommand{\la}{\langle}
\newcommand{\ra}{\rangle}
\newcommand{\Ind}{{\rm{Ind}}}
\newcommand{\Aut}{{\rm{Aut}}}
\newcommand{\Equi}{{\rm{Equi}}}

\newcommand{\Pic}{{\rm{Pic}}}

\newcommand{\Int}{{\rm{Int}}}
\newcommand{\Ker}{{\rm{Ker}}}
\newcommand{\Ima}{{\rm{Im}}}

\title{The Picard groups of unital inclusions of unital $C^*$-algebras induced by involutive equivalence bimodules}
\author{Kazunori Kodaka}
\address{Department of Mathematical Sciences, Faculty of Science, Ryukyu
\endgraf
University, Nishihara-cho, Okinawa, 903-0213, Japan}
\address{\sl{E-mail address}: \rm{kodaka@math.u-ryukyu.ac.jp}}



\begin{document}
\maketitle
\begin{abstract}
Let $A$ be a unital $C^*$-algebra and $X$ an invulutive $A-A$-equivalence bimodule.
Let $A\subset C_X$ be the unital inclusion of unital $C^*$-algebras induced by $X$.
We suppose that $A' \cap C_X =\BC 1$. We shall compute the Picard group of the unital
inclusion $A\subset C_X$.
\end{abstract}
 
\section{Introduction}\label{sec:intro} Let $A$ be a unital $C^*$-algebra and $X$ an
$A-A$-equivalence bimodule. Following \cite {KT0:involution}, we say that $X$ is
\sl
involutive
\rm
if there exists a conjugate linear map
$x\mapsto x^{\natural}$ on $X$ such that
\newline
(1) $(x^{\natural})^{\natural}=x$, \quad $x\in X$,
\newline
(2) $(a\cdot x\cdot b)^{\natural}=b^* \cdot x^{\natural}\cdot a^*$, \quad $x\in X$, \quad $a, b\in A$,
\newline
(3) ${}_A \la x, y^{\natural} \ra =\la x^{\natural} , y \ra_A$, \quad $x, y\in X$,
\newline
where ${}_A \la -, - \ra$ and $\la -, - \ra_A$ are the left and the right $A$-valued inner products
on $X$, respectively. We call the above conjugate linear map an
\sl
involution
\rm
on $X$. For each $A-A$-equivalence bimodule $X$, $\widetilde{X}$ denotes its dual
$A-A$-equivalence bimodule. For each $x\in X$, $\widetilde{x}$ denotes the
element in $\widetilde{X}$ induced by $x$. For each involutive $A-A$-equivalence bimodule
$X$, let $L_X$ be the linking $C^*$-algebra for $X$ defined in Brown, Green and Rieffel
\cite {BGR:linking}. Following \cite {KT0:involution}, we define the $C^*$-subalgebra $C_X$ of $L_X$ by
$$
C_X =\{ \begin{bmatrix} a & x \\
\widetilde{x^{\natural}} & a \end{bmatrix} \, | \, a\in A, \, x\in X \} .
$$
We regard $A$ as a $C^*$-subalgebra of $C_X$, that is,
$A=\{\begin{bmatrix} a & 0 \\ 0 & a \end{bmatrix} \, | \, 
a\in A \}$.
\par
In \cite {Kodaka:Picard}, we defined the Picard of unital inclusion of unital $C^*$-algebras $A\subset C$.
We denote it by $\Pic(A, C)$. In this paper, we shall compute $\Pic(A, C_X )$ under the assumption
that $A' \cap C_X =\BC 1$. Let us explain the strategy of computing $\Pic(A, C_X )$. Let $f_A$ be
the homomorphism of $\Pic(A, C_X )$ to $\Pic(A)$ defined in \cite {Kodaka:Picard}, where $\Pic(A)$ is the
Picard group of $A$. We compute $\Ker f_A$ and $\Ima f_A$, the kernel of $f_A$ and the image of
$f_A$, respectively and we construct a homomorphism $g_A$ of $\Ima f_A$ to $\Pic(A, C_X )$ with
$f_A \circ g_A =\id_{\Pic(A)}$. We can compute $\Pic(A, C_X )$ in the above way.

\section{Preliminaries}\label{sec:pre} We recall the definition of the Picard group for
a unital inclusion of unital $C^*$-algebras $A\subset C$. Let $Y$ be a $C-C$-equivalence bimodule and $X$ its
closed subspace satisfying Conditions (1), (2) in \cite [Definition 2.1]{KT4:morita}. Let $\Equi(A, C)$
be the set of all such pairs $(X, Y)$ as above. We define an equivalence relation $``\sim "$ as follows:
For $(X, Y), (Z, W)\in \Equi(A, C)$, $(X, Y)\sim (Z, W)$ in $\Equi(A, C)$ if and only if there is a $C-C$-equivalence
bimodule isomorphism $\Phi$ of $Y$ onto $W$ such that the restriction of $\Phi$ to $X$, $\Phi|_X$ is
an $A-A$-equivalence bimodule isomorphism $X$ onto $Z$. We denote by $[X, Y]$, the
equivalence class of $(X, Y)$ in $\Equi(A, C)$. Let $\Pic(A, C)=\Equi(A, C)/\! \sim$.
We define the product in $\Pic(A, C)$ as follows: For $(X, Y), (Z, W)\in\Pic(A, C)$
$$
[X, Y][Z, W]=[X\otimes_A Z \, , \, Y\otimes_C W],
$$
where the $A-A$-equivalence bimodule $X\otimes_A Z$ is identified with the closed subspace
$``X\otimes_C Z"$ of $Y\otimes_C W$ by \cite [Lemma 3.1]{Kodaka:Picard} and
$``X\otimes_C Z"$ is defined by the closure of linear span of the set
$$
\{x\otimes z\in Y\otimes_C W \, | \, x\in X, \, z\in Z \}
$$
by \cite {Kodaka:Picard} and easy computations, $Y\otimes_C W$ and its closed
subspace $X\otimes_A Z$ satisfy Conditions (1), (2) in \cite [Definition 2.1]{KT4:morita}
and $\Pic (A, C)$ is a group. We regard $(A, C)$ as an element in $\Equi (A, C)$ in the
evident way. Then $[A, C]$ is unit element in $\Equi (A, C)$ in $\Pic(A, C)$. For any element
$(X, Y)\in\Equi(A, C)$, $(\widetilde{X}, \widetilde{Y})\in\Equi(A, C)$ and $[\widetilde{X}, \widetilde{Y}]$
is the inverse element of $[X, Y]$ in $\Pic(A, C)$. We call the group $\Pic (A, C)$ defined in the
above, the Picard group of the unital inclusion of unital $C^*$-algebras $A\subset C$.
\par
Let $f_A$ be the homomorphism of $\Pic(A, C)$ to $\Pic(A)$ defined by
$$
f_A ([X, Y])=[X]
$$
for any $(X, Y)\in \Equi(A, C)$.

\section{Kernel}\label{sec:kernel} Let $A$ be a unital $C^*$-algebra and $X$ an involutive
$A-A$-equivalence bimodule. Let $A\subset C_X$ be the unital inclusion of unital $C^*$-algebras induced by
$X$ and we suppose that $A' \cap C_X =\BC1$. Let $f_A$ be the homomorphism of $\Pic(A, C_X )$
to $\Pic(A)$ defined by
$$
f_A ([M, N])=[M]
$$
for any $(M, N)\in \Equi(A, C_X )$. In this section, we compute $\Ker f_A$. Let $(M, N)\in \Equi (A, C_X )$.
We suppose that $[M, N]\in\Ker f_A$. Then $[M]=[A]$ in $\Pic(A)$ and by \cite [Lemma 7.5]{Kodaka:Picard},
there is a $\beta\in \Aut_0 (A, C_X )$ such that
$$
[M, N]=[A, N_{\beta}]
$$
in $\Pic(A, C_X)$ where $\Aut_0 (A, C_X )$ is the group of all automorphisms $\beta$
such that $\beta(a)=a$ for any $a\in A$ and $N_{\beta}$ is the
$C_X -C_X$-equivalence bimodule induced by $\beta$ which is defined in \cite [Section 2]{Kodaka:Picard}.
By the above discussions, we obtain
the following lemma.

\begin{lemma}\label{lem:ker1} With the above notation,
$$
\Ker f_A =\{[A, N_{\beta}]\in\Pic(A, C_X ) \, | \, \beta\in \Aut_0 (A, C_X ) \} .
$$
\end{lemma}

Let $\Aut(A, C_X )$ be the group of all automorphisms $\alpha$ of $C_X$ such that the
restriction of $\alpha$ to $A$, $\alpha|_A$ is an automorphism of $A$. Then $\Aut_0 (A, C_X )$
is a normal subgroup of $\Aut(A, C_X )$. Let $\pi$ be the homomorphism of $\Aut(A, C_X )$ to
$\Pic(A, C_X )$ defined by
$$
\pi(\alpha)=[M_{\alpha}, N_{\alpha}]
$$
for any $\alpha\in\Aut(A, C_X )$, where $(M_{\alpha}, N_{\alpha})$ is the element
in $\Equi(A, C_X )$ induced by $\alpha\in \Aut(A, C_X )$ (See \cite [Section 3]{Kodaka:Picard}).
By Lemma \ref{lem:ker1}, $\pi(\Aut_0 (A, C_X ))=\Ker f_A$ and \cite [Lemma 3.4]{Kodaka:Picard},
$$
\Ker \pi \cap \Aut_0 (A, C_X )=\Int (A, C_X )\cap\Aut_0 (A, C_X ) ,
$$
where $\Int(A, C_X )$ is the group of all $\Ad(u)$ such that $u$ is a unitary element in $A$. Hence
\begin{align*}
\Ker \pi \cap \Aut_0 (A, C_X ) & =\{\Ad(u)\in \Aut_0 (A, C_X ) \, | \, \text{$u$ is a unitary element in $A$} \} \\
& =\{\Ad(u)\in\Aut_0 (A, C_X ) \, | \, \text{$u$ is a unitary element in $A' \cap A$} \} .
\end{align*}
Since $A' \cap C_X =\BC1$, $A' \cap A=\BC1$. Thus we can see that $\Ker \pi \cap \Aut_0 (A, C_X )=\{1\}$.
It follows that we can obtain that the following lemma.

\begin{lemma}\label{lem:ker2} With the above notation, $\Ker f_A \cong \Aut_0 (A, C_X )$.
\end{lemma}

Let ${}_A \Aut_A^{\natural}(X)$ be the group of all involutive $A-A$-equivalence bimodule automorphisms
of $X$. Let $E^A$ be the conditional expectation from $C_X$ onto $A$ defined by
$$
E^A (\begin{bmatrix} a & x \\
\widetilde{x^{\natural}} & a \end{bmatrix})
=\begin{bmatrix} a & 0 \\
0 & a \end{bmatrix}
$$
for any $a\in A$, $x\in X$. Then $E^A$ is of Watatani index-finite type by 
\cite [Lemma 3.4]{KT0:involution}.

\begin{lemma}\label{lem:composition} Withe the above notation, $E^A =E^A \circ\beta$ for any
$\beta\in\Aut_0 (A, C_X )$.
\end{lemma}
\begin{proof} Let $\beta\in \Aut_0 (A, C_X )$. Then $E^A \circ\beta$ is also a conditional expectation
from $C_X$ onto $A$. Since $A' \cap C_X =\BC 1$, by Watatani \cite [Proposition 1.4.1]{Watatani:index},
$E^A =E^A \circ\beta$.
\end{proof}

\begin{lemma}\label{lem:unique} With the above notation, for any $\beta\in Aut_0 (A, C_X )$,
there is the unique $\theta\in \Aut_0^{\natural}(X)$ such that
$$
\beta(\begin{bmatrix} a & x \\
\widetilde{x^{\natural}} & a \end{bmatrix})
=\begin{bmatrix} a & \theta(x) \\
\widetilde{\theta(x^{\natural})} & a \end{bmatrix}
$$
for any $a\in A$, $x\in X$.
\end{lemma}
\begin{proof} For any $x\in X$, let
$$
\beta(\begin{bmatrix} 0 & x \\
\widetilde{x^{\natural}} & 0 \end{bmatrix})
=\begin{bmatrix} b & y \\
\widetilde{y^{\natural}} & b \end{bmatrix} ,
$$
where $b\in A$, $y\in X$. Then by Lemma \ref{lem:composition},
$$
\begin{bmatrix} b & 0 \\
0 & b \end{bmatrix}
=(E^A \circ\beta)(\begin{bmatrix} 0 & x \\
\widetilde{x^{\natural}} & 0 \end{bmatrix})
=E^A (\begin{bmatrix} 0 & x \\
\widetilde{x^{\natural}} & 0 \end{bmatrix})
=\begin{bmatrix} 0 & 0 \\
0 & 0 \end{bmatrix} .
$$
Hence $b=0$. Thus
$$
\beta(\begin{bmatrix} 0 & x \\
\widetilde{x^{\natural}} & 0 \end{bmatrix})
=\begin{bmatrix} 0 & y \\
\widetilde{y^{\natural}} & 0 \end{bmatrix} .
$$
We define a map $\theta$ on $X$ by
$$
\theta(x)=y ,
$$
where $y$ is the element in $X$ defined as above. Then clearly $\theta$ is linear and since
$$
\beta(\begin{bmatrix} 0 & x^{\natural} \\
\widetilde{x} & 0 \end{bmatrix})
=\beta(\begin{bmatrix} 0 & x \\
\widetilde{x^{\natural}} & 0 \end{bmatrix})^*
=\begin{bmatrix} 0 & y^{\natural} \\
\widetilde{y} & 0 \end{bmatrix} ,
$$
we obtain that
$$
\theta(x^{\natural})=y^{\natural}=\theta(x)^{\natural} .
$$
Hence $\theta$ preserves the involution $\natural$.
Also, for any $a\in A$, $x\in X$,
\begin{align*}
\begin{bmatrix} 0 & \theta(a\cdot x) \\
\theta(\widetilde{(a\cdot x)^{\natural}}) & 0 \end{bmatrix}
& =\beta(\begin{bmatrix} 0 & a\cdot x \\
a\cdot \widetilde{x^{\natural}} & 0 \end{bmatrix})
=\beta(\begin{bmatrix} a & 0 \\
0 & a \end{bmatrix}
\begin{bmatrix} 0 & x \\
\widetilde{x^{\natural}} & 0 \end{bmatrix}) \\
& =\begin{bmatrix} a & 0 \\
0 & a \end{bmatrix}
\begin{bmatrix} 0 & \theta(x) \\
\theta(\widetilde{x^{\natural}}) & 0 \end{bmatrix}
=\begin{bmatrix} 0 & a\cdot \theta(x) \\
a\cdot \theta(\widetilde{x^{\natural}}) & 0 \end{bmatrix} .
\end{align*}
Hence $\theta(a\cdot x)=a\cdot\theta(x)$ for any $a\in A$, $x\in X$.
Similarly $\theta(x\cdot a)=\theta(x)\cdot a$ for any $a\in A$, $x\in X$.
Furthermore, for any $x, y\in X$,
\begin{align*}
\begin{bmatrix} {}_A \la \theta(x) \, , \, \theta(y) \ra & 0 \\
0 & {}_A \la \theta(x) \, , \, \theta(y) \ra \end{bmatrix}
& =\begin{bmatrix} 0 & \theta(x) \\
\widetilde{\theta(x)^{\natural}} & 0 \end{bmatrix}
\begin{bmatrix} 0 & \theta(y)^{\natural} \\
\widetilde{\theta(y)} & 0 \end{bmatrix} \\
& =\beta(\begin{bmatrix} 0 & x \\
\widetilde{x^{\natural}} & 0 \end{bmatrix}
\begin{bmatrix} 0 & y^{\natural} \\
\widetilde{y} & 0 \end{bmatrix}) \\
& =\beta(\begin{bmatrix} {}_A \la x \, , \, y \ra & 0 \\
0 & {}_A \la x \, , \, y \ra \end{bmatrix}) \\
& =\begin{bmatrix} {}_A \la x \, , \, y \ra & 0 \\
0 & {}_A \la x \, , \, y \ra \end{bmatrix} .
\end{align*}
Hence ${}_A \la \theta(x) \, ,\, \theta(y) \ra={}_A \la x \, , \, y \ra$ for any $x, y\in X$.
Similarly for any $x, y\in X$,
$$
\begin{bmatrix} \la \theta(x) \, , \, \theta(y) \ra_A & 0 \\
0 & \la \theta(x) \, , \, \theta(y) \ra_A \end{bmatrix}
=\begin{bmatrix}  \la x \, , \, y \ra_A & 0 \\
0 &  \la x \, , \, y \ra_A \end{bmatrix} .
$$
Hence $\la \theta(x) \, ,\, \theta(y) \ra_A =\la x \, , \, y \ra_A$ for any $x, y\in X$.
Thus $\theta\in\!{}_A \Aut_A^{\natural}(X)$. Next, let $\theta\in {}_A \Aut_A^{\natural}(X)$.
Then let $\beta$ be a map on $C_X$ defined by
$$
\beta(\begin{bmatrix} a & x \\ \widetilde{x^{\natural}} & a \end{bmatrix})
=\begin{bmatrix} a & \theta(x) \\
\widetilde{\theta(x)^{\natural}} & a \end{bmatrix}
$$
for any $a\in A$, $x\in X$. Then by easy computations, $\beta\in \Aut_0 (A, C_X )$.
Therefore, we obtain the conclusion.
\end{proof}

\begin{cor}\label{cor:morphism} With the above notation, $\Aut_0 (A, C_X )\cong {}_A \Aut_A^{\natural}(X)$.
\end{cor}
\begin{proof} This is immediate by Lemma \ref{lem:unique}.
\end{proof}

Let ${}_A \Aut_A (X)$ be the group of all $A-A$-equivalence bimodule automorphisms.
Since $X$ is an $A-A$-equivalence bimodule, ${}_A \Aut_A (X)$ is isomorphic to $U(A' \cap A)$, 
the group of all unitary elements in $A' \cap A$. Since $A' \cap C_X =\BC1$, $U(A' \cap A)=\BT1$,
where $\BT$ is the 1-dimensional torus. Hence ${}_A \Aut_A^{\natural}(X)$ is isomorphic to a subgroup of
$\BT$. But since $\lambda1$ preserves the operation $\natural$ for any $\lambda\in\BT$,
${}_A \Aut_A^{\natural}(X)\cong \BT1$. By the above discussions, we can obtain the following
proposition.

\begin{prop}\label{prop:ker3} With the above notation, $\Ker f_A \cong \BT1$.
\end{prop}
\begin{proof} This is immediate by Lemma \ref{lem:ker2}, Corollary \ref{cor:morphism}
and the above discussions.
\end{proof}

\section{A result on strongly Morita equivalent unital inclusions of unital $C^*$-algebras}
\label{sec:construction}
In this section, we shall prove the following result: Let $H$ be a finite dimensional $C^*$-Hopf algebra
and $H^0$ its dual $C^*$-Hopf algebra. Let $(\rho, u)$ and $(\sigma, v)$ be twisted
coactions of $H^0$ on unital $C^*$-algebras $A$ and $B$,
respectively. Let $A\subset A\rtimes_{\rho, u}H$ and $B\subset B\rtimes_{\sigma, v}H$ be unital
inclusions of unital $C^*$-algebras. We suppose that they are strongly Morita equivalent with respect to
an $A\rtimes_{\rho, u}H-B\rtimes_{\sigma,v}H$-equivalence bimodule $Y$ and its closed subspace $X$.
And we suppose that $A' \cap (A\rtimes_{\rho, u}H)=\BC1$. Then there are a twisted coaction $(\gamma, w)$
of $H^0$ on $B$ and a twisted coaction $\lambda$ of $H^0$ on $X$ satisfying the following:
\newline
(1) $(\rho, u)$ and $(\gamma, w)$ are strongly Morita equivalent with respect to
$\lambda$,
\newline
(2) $B\rtimes_{\sigma, v}H=B\rtimes_{\gamma, w}H$,
\newline
(3) $Y\cong X\rtimes_{\lambda}H$ as $A\rtimes_{\rho, u}H-B\rtimes_{\sigma, v}H$-
equivalence bimodules.
\par
In the next section, we shall use this
result in the case of $\BZ_2$-actions, where $\BZ_2 =\BZ/2\BZ$.
We shall use the results in \cite {KT5:Hopf} in order to prove the above result. First we recall
\cite {KT5:Hopf}.
\par
Let $H$ be a finite dimensional $C^*$-Hopf algebra. We denote
its comultiplication, counit and antipode by $\Delta$, $\epsilon$, and $S$, respectively.
We shall use Sweedler's notation $\Delta(h)=h_{(1)}\otimes h_{(2)}$ for any $h\in H$ which
surppresses a possible summation when we write comultiplications. We denote by $N$ the
dimension of $H$. Let $H^0$ be the dual $C^*$-Hopf algebra of $H$. We denote its comultiplication,
counit and antipode by $\Delta^0$, $\epsilon^0$ and $S^0$, respectively. There is the distinguished
projection $e$ in $H$. We note that $e$ is the Haar trace on $H^0$. Also, there is the distinguished
projection $\tau$ in $H^0$ which is the Haar trace on $H$. Since $H^0$ is finite dimensional,
$H^0 \cong \oplus_{k=1}^K M_{d_k}(\BC)$ as $C^*$-algebras, where $M_n (\BC)$ is the $n\times n$-
matrix algebra over $\BC$.
Let
$$
\{w_{ij}^k \, | \, k=1,2,\dots, K, \, i,j=1,2,\dots, d_k\}
$$
be a basis of $H$ satisfying Szyma\'nski and
Peligrad's \cite [Theorem 2.2,2]{SP:saturated}, which is called a system of
\sl
comatrix units
\rm
of $H$, that is, the dual basis of a system of matrix units of $H^0$. 
\par
Let $A$ be a unital $C^*$-algebra and $(\rho, u)$ a twisted coaction of $H^0$ on $A$, that is,
$\rho$ is a weak coaction of $H^0$ on $A$ and $u$ is a unitary element in $A\otimes H^0 \otimes H^0$
satisfying that
\newline
(1) $(\rho\otimes\id)\circ\rho=\Ad(u)\circ(\id\otimes\Delta^0 )\circ\rho$,
\newline
(2) $(u\otimes 1^0 )(\id\otimes\Delta^0 \otimes\id)(u)=(\rho\otimes \id\otimes\id)(u)(\id\otimes\id\otimes
\Delta^0 )(u)$,
\newline
(3) $(\id\otimes h\otimes\epsilon^0 )(u)=(\id\otimes\epsilon^0 \otimes h)(u)=\epsilon^0 (h)1$ for any
$h\in H$.
\newline
For a twisted coaction $(\rho, u)$ of $H^0$ on $A$, we can consider the twisted action of $H$
on $A$ and its unitary element $\widehat{u}$ defined by
$$
h\cdot_{\rho, u}x =(\id\otimes h)(\rho(x)) , \quad
\widehat{u}(h, l)=(\id\otimes h\otimes l)(u)
$$
for any $x\in A$, $h, l\in H$. We call it the twisted action of $H$ on $A$ induced by
$(\rho, u)$. Let $A\rtimes_{\rho, u}H$ be the twisted crossed product of $A$ by the
twisted action of $H$ induced by $(\rho, u)$. Let $x\rtimes_{\rho, u}h$ be the
element in $A\rtimes_{\rho, u}H$ induced by $x\in A$ and $h\in H$. Let $\widehat{\rho}$ be
the dual coaction of $H$ on $A\rtimes_{\rho, u}H$ defined by
$$
\widehat{\rho}(x\rtimes_{\rho, u}h)=(x\rtimes_{\rho, u}h_{(1)})\otimes h_{(2)}
$$
for any $x\in A$, $h\in H$. Let $E_1^{\rho, u}$ be the canonical conditional expectation
from $A\rtimes_{\rho, u}H$ onto $A$ defined by
$$
E_1^{\rho, u}(x\rtimes_{\rho, u}h)=\tau(h)x
$$
for any $x\in A$, $h\in H$. Let $\Lambda$ be the set of all triplets $(i, j, k)$, where
$i, j=1,2,\dots,d_k$ and $k=1,2,\dots, K$ with $\sum_{k=1}^K d_k^2 =N$. Let
$W_I^{\rho}=\sqrt{d_k}\rtimes_{\rho, u}w_{ij}^k$ fo any $I=(i, j, k)\in\Lambda$.
By \cite [Proposition 3.18]{KT1:inclusion}, $\{(W_I^{\rho*}, W_I^{\rho})\}_{I\in\Lambda}$
is a quasi-basis for $E_1^{\rho, u}$.
\par
Let $A$ and $B$ be unital $C^*$-algebras and let $(\rho, u)$ and $(\sigma, v)$ be twisted
coactions of $H^0$ on $A$ and $B$, respectively. Let $A\rtimes_{\rho, u}H$ and $B\rtimes_{\sigma,v}H$
be the twisted crossed products of $A$ and $B$ by $(\rho, u)$ and $(\sigma, v)$, respectively.
We denote them by $C$ and $D$, respectively. Then we obtain unital inclusions of unital
$C^*$-algebras, $A\subset C$ and $B\subset D$. We suppose that $A\subset C$ and $B\subset D$
are strongly Morita equivalent with respect to a $C-D$-equivalence bimodule $Y$ and its
closed subspace $X$. We also suppose that $A' \cap C=\BC 1$. Then $B' \cap D=\BC 1$
by \cite [Lemma 10.3]{KT4:morita}. And by \cite [Theorem 2.9]{KT4:morita}, there are a conditional
expectation $F^B$ from $D$ onto $B$ and a conditional expectation $E^X$ from $Y$ onto
$X$ with respect to $E_1^{\rho, u}$ and $F^B$ satisfying Conditions (1)--(6) in
\cite [Definition 2.4]{KT4:morita}. Since $B' \cap D=\BC 1$, by Watatani
\cite [Proposition 1.4.1]{Watatani:index}, $F^B =E_1^{\sigma, v}$, the canonical conditional
expectation from $D$ onto $B$. Furthermre, by \cite [Section 6]{KT4:morita},
we can see that the unital inclusions of unital $C^*$-algebras, $C\subset C_1$ and $D\subset D_1$
are strongly Morita equivalent with respect to the $C_1 -D_1$-equivalence bimodule $Y_1$ and
its closed subspace $Y$, where $C_1=C\rtimes_{\widehat{\rho}}H^0$
and $D_1 =D\rtimes_{\widehat{\sigma}}H^0$ and $\widehat{\rho}$ and $\widehat{\sigma}$ are
the dual coactions of $(\rho, u)$ and $(\sigma, v)$, respectively. And $Y_1$ is defined as follows:
We regard $C$ and $D$ as a $C_1 -A$-equivalence bimodule and a $D_1 -B$-equivalence
bimodule in the usual way as in \cite [Section 4]{KT4:morita}, respectively.
Let $Y_1 =C\otimes_A X\otimes _B \widetilde{D}$. Let $E^Y$ be the conditional expectation
from $Y_1$ onto $Y$ with respect to $E_2^{\rho, u}$ and $E_2^{\sigma, v}$ defined by
$$
E^Y (c\otimes x \otimes\widetilde{d})=\frac{1}{N}c\cdot x\cdot d^*
$$
for any $c\in C$, $d\in D$, $x\in X$, where $E_2^{\rho, u}$ and $E_2^{\sigma, v}$ are the
canonical conditional expectations from $C_1$ and $D_1$ onto $C$ and $D$, respectively.
We regard $Y$ as a closed subspace of $Y_1$ by the injective linear map $\phi$ from $Y$ into
$Y_1$ defined by
$$
\phi(y)=\sum_{I, J\in \Lambda}W_I^{\rho *}\otimes E^X (W_I^{\rho}\cdot y\cdot W_J^{\sigma *})
\otimes \widetilde{W_J^{\sigma*}}
$$
for any $y\in Y$. By \cite [Sections 3 and 4]{KT5:Hopf}, there are a coaction $\beta$ of $H$ on $D$
and a coaction $\mu$ of $H$ on $Y$ such that $(C, D, Y, \widehat{\rho}, \beta, H)$ is a covariant
system, that is, $\mu$ is a coaction of $H$ on $Y$ with respect to $(C, D, \widehat{\rho}, \beta)$.
We define the action of $H^0$ on $Y$ induced by $\mu$ as follows: For any $\psi\in H^0$, $y\in Y$,
$$
\psi\cdot_{\mu}y=NE^Y ((1\rtimes_{\rho, u}1\rtimes_{\widehat{\rho}}\psi)\cdot \phi(y)\cdot
(1\rtimes_{\sigma, v}1\rtimes_{\widehat{\sigma}}\tau))
$$
By \cite [Remark 3.1]{KT5:Hopf}, for any $\psi\in H^0$, $y\in Y$,
$$
\psi\cdot_{\mu}y =\sum_{I\in\Lambda}[\psi\cdot_{\widehat{\rho}}W_I^{\rho*}]\cdot E^X (W_I^{\rho}\cdot y) .
$$
Also, we define the action of $H^0$ on $D$ induced by $\beta$ as follows:
For any $\psi\in H^0$, $y, z\in Y$,
$$
\psi\cdot_{\beta}\la y, z \ra_D =\la S^0 (\psi_{(1)}^*)\cdot_{\mu} y \, , \, \psi_{(2)}\cdot_{\mu}z \ra_D ,
$$
where we regard $D$ as the linear span of the set
$\{ \la y, z \ra_D \, | \, y, z\in Y \} $.

\begin{lemma}\label{lem:action} For any $y\in Y$, $\tau\cdot_{\mu}y=E^X (y)$.
\end{lemma}
\begin{proof} By routine computations, we obtain the lemma. Indeed, by
\cite [Remark 3.1]{KT5:Hopf}, for any $y\in Y$,
\begin{align*}
\tau\cdot_{\mu}y & =\sum_{I\in \Lambda}[\tau\cdot_{\widehat{\rho}}W_I^{\rho*}]
\cdot E^X (W_I^{\rho}\cdot y) \\
& =\sum_{i, j, k}[\tau\cdot_{\widehat{\rho}}(\sqrt{d_k}\rtimes_{\rho, u}w_{ij}^k )^* ]
\cdot E^X ((\sqrt{d_k}\rtimes_{\rho, u}w_{ij}^k)\cdot y) \\
& =\sum_{i, j, j_1, j_2 , j_3 , k}[\tau\cdot_{\widehat{\rho}}(\widehat{u}(S(w_{j_1 j_2 }^k )\, , \, w_{ij_1 }^k )^*
[w_{j_2 j_3 }^{k*}\cdot_{\widehat{\rho}} \sqrt{d_k}]\rtimes_{\rho, u}w_{j_3 j}^{k*})] \\
& \cdot E^X ((\sqrt{d_k}\rtimes_{\rho, u}w_{ij}^k ) \cdot y) \\
& =\sum_{i, j, j_1 , j_2 , k}\sqrt{d_k }[\tau\cdot_{\widehat{\rho}}
(\widehat{u}(S(w_{j_1 j_2 }^k ) \, , \, w_{ij_1}^k) ^* \rtimes_{\rho, u}w_{j_2 j}^{k*})]
\cdot E^X ((\sqrt{d_k} \rtimes_{\rho, u}w_{ij}^k )\cdot y) \\
& =\sum_{i, j, j_1 , j_2 , j_3 , k}\sqrt{d_k}(\widehat{u}(S(w_{j_1 j_2}^k ) \, , \, w_{ij_1}^k )^*
\rtimes_{\rho, u}w_{j_2 j_3}^{k*})\tau(w_{j_3 j}^{k*}) \\
& \cdot E^X ((\sqrt{d_k}\rtimes_{\rho, u}w_{ij}^k )\cdot y) \\
& =\sum_{i, j, j_1 , j_2 , k}\sqrt{d_k}(\widehat{u}(S(w_{j_1 j_2}^k )\, , \, w_{ij_1}^k )^* \rtimes_{\rho, u}
\tau(w_{j_2 j}^{k*}))\cdot E^X ((\sqrt{d_k}\rtimes_{\rho, u}w_{ij}^k )\cdot y) .
\end{align*}
Since $\tau \circ S^0 =\tau$, $\tau(w_{j_2 j}^{k*})=\tau(S(w_{jj_2}^k ))=(\tau\circ S^0 )(w_{jj_2 }^k )=\tau(w_{jj_2 }^k )$.
Hence
\begin{align*}
\tau\cdot_{\mu}y & =\sum_{i, j, j_1 , j_2 , k}\sqrt{d_k}
(\widehat{u}(S(w_{j_1 j_2 }^k ) \, , \, w_{ij_1 }^k )^* 
\rtimes_{\rho, u}1 )\cdot E^X ((\sqrt{d_k }\rtimes_{\rho, u}w_{ij}^k \tau (w_{jj_2 }^k ))\cdot y) \\
& =\sum_{i, j_1 , j_2 , k}\sqrt{d_k}(\widehat{u}(S(w_{j_1 j_2}^k ) \, ,\, w_{i j_1}^k )^* \rtimes_{\rho, u}1 )
\cdot E^X ((\sqrt{d_k}\rtimes_{\rho, u}\tau(w_{ij_2}^k ))\cdot y) \\
& =\sum_{i, j_1 , j_2 , k}d_k (\widehat{u}(S(w_{j_1 j_2}^k ) \, , \, w_{ij_1 }^k \overline{\tau(w_{ij_2}^k )})^*
\rtimes_{\rho, u}1)\cdot E^X (y) .
\end{align*}
Since $\tau=\tau^* $, $\overline{\tau(w_{ij_2}^k )}=\overline{\tau^* (w_{ij_2}^k)}=\tau(S(w_{ij_2}^{k*}))
=\tau(w_{j_2 i}^k )$.
Hence
\begin{align*}
\tau\cdot_{\mu}y & =\sum_{i, j_1 , j_2 , k}d_k (\widehat{u}(S(w_{j_1 j_2 }^k ) \, , \,
w_{ij_1}^k \tau(w_{j_2 i}^k ))^* \rtimes_{\rho, u}1)\cdot E^X (y) \\
& =\sum_{j_1 , j_2 , k}d_k (\widehat{u}(S(w_{j_1 j_2}^k ) \, , \, \tau(w_{j_2 j_1}^k  ))^* \rtimes_{\rho, u}1)
\cdot E^X (y) \\
& =\sum_{j_1 ,k}d_k (\widehat{u}(S(\tau(w_{j_1 j_1 }^k )) \, , \, 1 )^* \rtimes_{\rho, u}1 )\cdot E^X (y) \\
& =\sum_{j_1 , k}d_k \overline{\epsilon(S(\tau(w_{j_1 j_1 }^k )))}\cdot E^X (y) .
\end{align*}
Since $\epsilon \circ S =\epsilon$, $\epsilon(S(\tau(w_{j_1 j_1 }^k )))=\tau(w_{j_1 j_1 }^k )1$.
Hence
\begin{align*}
\tau\cdot_{\mu} y & =\sum_{j_1 k}d_k \overline{\tau(w_{j_1 j_1}^k )}E^X(y)
=\overline{\sum_{j_1 k}d_k \tau(w_{j_1 j_1}^k )}E^X (y) =N\tau(e)E^X (y) \\
& =E^X (y)
\end{align*}
since $e=\frac{1}{N}\sum_{j,k}d_k w_{jj}^k $ . Therefore, we obtain the conclusion.
\end{proof}

We recall that the unital inclusions of unital $C^*$-algebras, $C\subset C_1$ and $D\subset D_1$
are strongly Morita equivalent with respect to $Y_1$ and its closed subspace $Y$. Also,
$C\subset C_1$ and $D\subset D\rtimes_{\beta}H^0$ are strongly Morita equivalent with
respect to the $C_1 -D\rtimes_{\beta}H^0$-equivalence bimodule $Y\rtimes_{\mu}H^0$
and its closed subspace $Y$, where $Y\rtimes_{\mu}H^0$ is the crossed product of $Y$ by
the coaction $\mu$ and it is a $C_1 -D\rtimes_{\beta}H^0$-equivalence bimodule
(See \cite {KT3:equivalence}). Hence the unital inclusions $D\subset D_1$ and
$D\subset D\rtimes_{\beta}H^0$ are strongly Morita equivalent with respect to
the $D_1 -D\rtimes_{\beta}H^0$-equivalence bimodule $\widetilde{Y_1}\otimes_{C_1}(Y\rtimes_{\mu}H^0 )$
and its closed subspace $\widetilde{Y}\otimes_C Y$. Then since $\widetilde{Y}\otimes_C Y$ is
isomorphic to $D$ as $D-D$-equivalence bimodule, we can see that there is an isomorphism
$\Psi$ of $D_1$ onto $D\rtimes_{\beta}H^0$ which is defined as follows: Since $Y$ is a $C-D$-
equivalence bimodule, there are elements $y_1 , \dots, y_n\in Y$ such that 
$\sum_{i=1}^n \la y_i , y_i \ra_D=1$. Let $\Psi$ be the map from $D_1$ to $D\rtimes_{\beta}H^0$
defined by
$$
\Psi(d)=\sum_{i, j} \la d\cdot \widetilde{y_i}\otimes y_i \, , \, \widetilde{y_j}\otimes y_j \ra_{D\rtimes_{\beta}H^0 }
$$
for any $d\in D_1$. By \cite [Section 5]{KT5:Hopf}, $\Psi$ is an isomorphism of $D_1$ onto
$D\rtimes_{\beta}H^0$ satisfying that $\Psi(d)=d$ for any $d\in D$ and that
$E_1^{\beta}\circ \Psi=E_2^{\sigma, v}$,
where $E_1^{\beta}$ is a canonical conditional expectation from $D\rtimes_{\beta}H^0$ onto $D$ and
$E_2^{\sigma, v}$ is the canonical conditional expectation from $D_1$ onto
$D$.

\begin{remark}\label{rem:dual} $\widetilde{Y}$ is a closed subspace of $\widetilde{Y_1}$ by the
inclusion $\widetilde{\phi}$ defined by
$$
\widetilde{\phi}(\widetilde{y})=\widetilde{\phi(y)}
$$
for any $y\in Y$. Also, $Y$ is a closed subspace
$Y\rtimes_{\mu}H^0$ by the inclusion defined by
$$
Y\longrightarrow Y\rtimes_{\mu}H^0\, : y\mapsto y\rtimes_{\mu}1^0 .
$$
\end{remark}

Let $e_B$ be the Jones projection in $D_1$ for the canonical conditional expectation
$E_1^{\sigma, v}$ from $D$ onto $B$. We identify $e_B$ with the projection $1\rtimes_{\widehat{\sigma}}\tau$
in $D_1$.

\begin{lemma}\label{lem:projection} With the above notation, $\Psi(e_B )=\Psi(1\rtimes_{\widehat{\sigma}}\tau)
=1\rtimes_{\beta}\tau$.
\end{lemma}
\begin{proof} The lemma can be proved by routine computations. Indeed,
we note that $\widetilde{Y}$ is regarded as a closed subspace of $\widetilde{Y_1}$
by the inclusion $\widetilde{\phi}$ and $Y$ is a regarded as a closed
subspace $Y\rtimes_{\beta}1^0$ of $Y\rtimes_{\beta}H^0$. Then
\begin{align*}
\Psi(e_B ) & =\Psi(1\rtimes_{\widehat{\sigma}}\tau)
=\sum_{i, j}\la (1\rtimes_{\widehat{\sigma}}\tau)\cdot (\widetilde{y_i}\otimes y_i ) \, , \,
\widetilde{y_j }\otimes y_j \ra_{D\rtimes_{\beta}H^0} \\
& =\sum_{i, j}\la [y_i \cdot (1\rtimes_{\widehat{\tau}}\tau)]^{\widetilde{}}\otimes y_i \, , \,
\widetilde{y_j}\otimes y_j \ra_{D\rtimes_{\beta}H^0 } \\
& =\sum_{i, j}\la y_i \, , \, \la [y_i \cdot (1\rtimes_{\widehat{\sigma}}\tau)]^{\widetilde{}} \, , \,
\widetilde{y_j} \ra_{C_1}\cdot y_j \ra_{D\rtimes_{\beta}H^0} .
\end{align*}
We note that $\widetilde{y_i}, \widetilde{y_j}\in \widetilde{Y}\subset \widetilde{Y_1}$ and that
$y_i , y_j \in Y=Y\rtimes_{\mu}1^0 \subset Y\rtimes_{\mu}H^0$. Hence
$$
\Psi(e_B )=\sum_{i, j}\la y_i \rtimes_{\mu}1^0 \, , \,
{}_{C_1} \la \phi(y_i )\cdot (1\rtimes_{\widehat{\sigma}}\tau)\, , \,
\phi(y_j )\ra \cdot y_j \rtimes_{\mu}1^0 \ra_{D\rtimes_{\beta}H^0} .
$$
Furthermore, let $\{(u_k , \, u_k^* )\}$ and $\{(v_l , \, v_l ^* )\}$ be quasis-bases for $E_1^{\rho, u}$ and
$E_1^{\sigma, v}$, respectively. Then
\begin{align*}
\phi(y_i )\cdot(1\rtimes_{\widehat{\sigma}}\tau) & =\phi(y_i )\cdot e_B
=\sum_{k, l}u_k \otimes E^X (u_k^* \cdot y_i \cdot v_l )\otimes\widetilde{v_l}\cdot e_B \\
& =\sum_{k, l}u_k \otimes E^X (u_k^* \cdot y_i \cdot v_l )\otimes \widetilde{E^B}(v_l ) \\
& =\sum_{k. l}u_k \otimes E^X (u_k^* \cdot y_i \cdot v_l E^B (v_l^* ))\otimes\widetilde{1_D} \\
& =\sum_{k}u_k \otimes E^X (u_k^* \cdot y_i )\otimes\widetilde{1_D} .
\end{align*}
Hence since $1\rtimes_{\widehat{\sigma}}\tau$ is a projection in $D_1$,
\begin{align*}
{}_{C_1} \la y_i \cdot (1\rtimes_{\widehat{\sigma}}\tau) \, , \, y_j \ra & =
{}_{C_1} \la y_i \cdot (1\rtimes_{\widehat{\sigma}}\tau )\, , \, y_j \cdot (1\rtimes_{\widehat{\sigma}}\tau) \ra \\
& =\sum_{k, l}{}_{C_1} \la u_k \otimes E^X (u_k^* \cdot y_i )\otimes \widetilde{1_D} \, , \,
u_l \otimes E^X (u_l^* \cdot y_j )\otimes \widetilde{1_D} \ra \\
& =\sum_{k, l}{}_{C_1}\la u_k \cdot {}_A \la E^X (u_k^* \cdot y_i )\otimes\widetilde{1_D} \, , \, 
E^X (u_l^* \cdot y_j )\otimes\widetilde{1_D} \ra \, , \, u_l \ra \\
& =\sum_{k, l}{}_{C_1} \la u_k \cdot {}_A \la E^X (u_k^* \cdot y_i )\cdot {}_B \la \widetilde{1_D} \, , \,
\widetilde{1_D} \ra \, , \, E^X (u_l^* \cdot y_j )\ra \, , \, u_l \ra \\
& =\sum_{k, l}{}_{C_1} \la u_k \cdot {}_A \la E^X (u_k^* \cdot y_i  )\, , \,
E^X (u_l^* \cdot y_j ) \ra \, , \, u_l \ra \\
& =\sum_{k, l}{}_{C_1} \la u_k E^A ({}_C \la u_k^* \cdot y_i \, ,\, E^X (u_l^* \cdot y_j )\ra) \, , \, u_l \ra \\
& =\sum_{k, l}{}_{C_1} \la u_k E^A (u_k^* {}_C \la y_i \, , \, E^X (u_l^* \cdot y_j )\ra ) \, , \, u_l \ra \\
& =\sum_l {}_{C_1} \la {}_C \la y_i \, , \, E^X (u_l^* \cdot y_j )\ra \, , \, u_l \ra \\
& =\sum_l {}_C \la y_i \, ,\, E^X (u_l^* \cdot y_j )\ra e_A u_l^* .
\end{align*}
Thus
\begin{align*}
\Psi(e_B )  & =\sum_{i, j, l}\la y_i \rtimes_{\mu}1^0 \, , \, {}_C \la y_i \, , \, E^X (u_l^* \cdot y_j )\ra
e_A u_l^* \cdot (y_j \rtimes_{\mu}1^0 )\ra_{D\rtimes_{\beta}H^0 } \\
& =\sum_{i, j, l}\la {}_C \la E^X (u_l^* \cdot y_j ) \, , \, y_i \ra \cdot (y_i\rtimes_{\mu}1^0 ) \, , \,
e_A u_l^* \cdot (y_j\rtimes_{\mu}1^0 ) \ra_{D\rtimes_{\beta}H^0} \\
& =\sum_{i, j, l}\la E^X (u_l^* \cdot y_j )\cdot \la y_i \, ,\, y_i \ra _D \, , \,
e_A u_l^* \cdot y_j \ra_{D\rtimes_{\beta}H^0} \\
& =\sum_{j, l} \la E^X (u_l^* \cdot y_j ) \, , \, e_A u_l^* \cdot y_j \ra_{D\rtimes_{\beta}H^0} \\
& =\sum_{j, l}\la u_l e_A \cdot E^X (u_l^* \cdot y_j ) \, , \, y_j \ra_{D\rtimes_{\beta}H^0} .
\end{align*}
Since we identify $e_A$ with $1\rtimes_{\widehat{\rho}}\tau$, we obtain that
\begin{align*}
u_l e_A \cdot E^X (u_l^* \cdot y_j ) & =(u_l \rtimes_{\widehat{\rho}}1^0 )(1\rtimes_{\widehat{\rho}}\tau)
\cdot (E^X (u_l^* \cdot y_j )\rtimes_{\mu}1^0 ) \\
& =(u_l \rtimes_{\widehat{\rho}}\tau)\cdot (E^X (u_l^* \cdot y_j )\rtimes_{\mu}1^0 ) .
\end{align*}
By Lemma \ref {lem:action}, $E^X (u_l^* \cdot y_j )=\tau' \cdot_{\mu}(u_l^* \cdot y_j )$,
where $\tau' =\tau$. Thus
\begin{align*}
u_l e_A \cdot E^X (u_l^* \cdot y_j ) & =(u_l \rtimes_{\widehat{\rho}}\tau)\cdot
[\tau' \cdot_{\mu}(u_l^* \cdot y_j )]\rtimes_{\mu}1^0 \\
& =u_l [\tau_{(1)}\cdot_{\mu}[\tau' \cdot_{\mu}(u_l^* \cdot y_j )]\rtimes_{\mu}\tau_{(2)}] \\
& =u_l [\tau' \cdot_{\mu}(u_l^* \cdot y_j )]\rtimes_{\mu}\tau \\
& =u_l E^X (u_l^* \cdot y_j )\rtimes_{\mu}\tau .
\end{align*}
It follows by \cite [Lemma 5.4]{KT4:morita} that
\begin{align*}
\Psi(e_B ) & =\sum_{j, l}\la u_l E^X (u_l^* \cdot y_j )\rtimes_{\mu}\tau \, , \, y_j \ra_{D\rtimes_{\beta}H^0}
=\sum_j \la y_j \rtimes_{\mu}\tau \, , \, y_j \rtimes_{\mu}1^0 \ra_{D\rtimes_{\beta}H^0} \\
& =\sum_j \tau_{(1)}^* \cdot_{\mu} \la y_j \, , \, y_j \ra_D \rtimes_{\beta}\tau_{(2)}^* 
= [\tau_{(1)}^* \cdot_{\mu} 1_D ]\rtimes_{\beta}\tau_{(2)}^* =1\rtimes_{\beta}\tau^* =1\rtimes_{\beta}\tau .
\end{align*}
Therefore we obtain the conclusion.
\end{proof}

Let $(Y\rtimes_{\mu}H^o )_{\Psi}$ be the $C_1 -D_1$-equivalence bimodule induced by
the $C_1 -D\rtimes_{\beta}H^0$-equivalence bimodule $Y\rtimes_{\mu}H^0$ and the isomorphism $\Psi$
of $D_1$ onto $D\rtimes_{\beta}H^0$. Let $E_1^{\mu}$ be the linear map from $Y\rtimes_{\mu}H^0$ onto
$Y$ defined by
$$
E_1^{\mu}(y\rtimes_{\mu}\psi)=\psi(e)y
$$
for any $y\in Y$, $\psi\in H^0$, where $y\rtimes_{\mu}\psi$ is the element
in $Y\rtimes_{\mu}H^0$ induced by $y\in Y$, $\psi\in H^0$.
Then $E_1^{\mu}$ is a conditional expectation from $Y\rtimes_{\mu}H^0$ onto $Y$
with respect to $E_2^{\rho, u}$ and $E_1^{\beta}$, the canonical conditional expectation
from $D\rtimes_{\beta}H^0$ onto $D$ by \cite [Proposition 4.1]{KT4:morita}.
Let $E_1^{\mu, \Psi}$ be the linear map from $(Y\rtimes_{\mu}H^0 )_{\Psi}$ onto $Y$
induced by $E_1^{\mu}$ and $\Psi$.

\begin{lemma}\label{lem:expectation} With the above notation, $E_1^{\mu,\Psi}$ is a
conditional expectation from $(Y\rtimes_{\mu}H^0 )_{\Psi}$ onto $Y$ with respect to
$E_2^{\rho, u}$ and $E_2^{\sigma, v}$.
\end{lemma}
\begin{proof} We shall show that Conditions (1)-(6) in \cite [Definition 2.4]{KT4:morita}
hold. Let $y, z\in Y$, $c\in C$, $d\in D$ and $\psi\in H^0$.
\newline
(1)
\begin{align*}
E_1^{\mu, \Psi}((c\rtimes_{\widehat{\rho}}\psi)\cdot y) & = E_1^{\mu, \Psi}((c\rtimes_{\widehat{\rho}}\psi)
\cdot(y\rtimes_{\mu}1^0 ))
=E_1^{\mu, \Psi}(c\cdot [\psi_{(1)}\cdot_{\mu}y]\rtimes_{\mu}\psi_{(2)}) \\
& =c\cdot [\psi_{(1)}\cdot_{\mu}y]\psi_{(2)}(e)=c\cdot\psi(e)y=\psi(e)c\cdot y .
\end{align*}
On the other hand,
$$
E_2^{\rho, u}(c\rtimes_{\widehat{\rho}}\psi)\cdot y=\psi(e)c\cdot y .
$$
Hence Condition (1) holds.
\newline
(2)
$$
E_1^{\mu, \Psi}(c\cdot (y\rtimes_{\mu}\psi))=E_1^{\mu, \Psi}((c\cdot y)\rtimes_{\mu}\psi)
=c\cdot y\psi(e)=\psi(e)c\cdot y .
$$
On the other hand,
$$
c\cdot E_1^{\mu, \Psi}(y\rtimes_{\mu}\psi)=c\cdot \psi(e)y=\psi(e)c\cdot y .
$$
Hence Condition (2) holds.
\newline
(3)
\begin{align*}
E_2^{\rho, u}({}_{C\rtimes_{\widehat{\rho}}H^0} \la y\rtimes_{\mu}\psi \, , \, z \ra) & =E_2^{\rho, u}
({}_{C\rtimes_{\widehat{\rho}}H^0} \la y\rtimes_{\mu}\psi\, , \, z\rtimes_{\mu}1^0 \ra) \\
& =E_2^{\rho, u}({}_C \la y \, , \, [S^0 (\psi_{(1)}^* )\cdot_{\mu}z]\ra\rtimes_{\widehat{\rho}}\psi_{(2)}) \\
& ={}_C \la y \, , \, [S^0 (\psi_{(1)}^* )\cdot_{\mu}z]\ra \psi_{(2)}(e) \\
& ={}_C \la y \, , \, [\overline{\psi_{(2)}(e)}S^0 (\psi_{(1)}^* )\cdot_{\mu} z]\ra \\
& ={}_C \la y\, , \,[S^0 (\psi_{(2)}^* )(e)S^0 (\psi_{(1)}^* )\cdot_{\mu} z] \ra \\
& ={}_C \la y\, , \, [e(S^0 (\psi^* )\cdot_{\mu} z]\ra \\
& ={}_C \la y \, , \, \overline{\psi(e)}z \ra=\psi(e)\,{}_C \la y , z \ra .
\end{align*}
On the other hand,
$$
{}_C \la E_1^{\mu, \Psi}(y\rtimes_{\mu}\psi)\, , \, z \ra ={}_C \la \psi(e)y \, , \, z \ra =\psi(e){}_C \la y \, , \, z \ra .
$$
Hence Condition (3) holds.
\newline
(4)
\begin{align*}
E_1^{\mu, \Psi}(y\cdot (d\rtimes_{\widehat{\sigma}}\psi)) & =
E_1^{\mu}(y\cdot \Psi(d\times_{\widehat{\sigma}}\psi))
=y\cdot E_1^{\beta}(\Psi(d\rtimes_{\widehat{\sigma}}\psi)) \\
& =y\cdot E_2^{\sigma, v}(d\rtimes_{\widehat{\sigma}}\psi) .
\end{align*}
Hence Condition (4) holds.
\newline
(5)
\begin{align*}
E_1^{\mu, \Psi}((y\rtimes_{\mu}\psi)\cdot d) & =E_1^{\mu}((y\rtimes_{\mu}\psi)\cdot\Psi(d))
=E_1^{\mu}(y\rtimes_{\mu}\psi)\cdot\Psi(d) \\
& =E_1^{\mu, \Psi}(y\rtimes_{\mu}\psi)\cdot d .
\end{align*}
Hence Condition (5) holds.
\newline
(6)
\begin{align*}
E_2^{\sigma, v}(\la y\rtimes_{\mu}\psi \, , \, z \ra_{D\rtimes_{\widehat{\sigma}}H^0}) & =
E_2^{\sigma, v}(\Psi^{-1}(\la y\rtimes_{\mu}\psi \, , \, z\rtimes_{\mu}1^0 \ra_{D\rtimes_{\beta}H^0})) \\
& =E_1^{\beta}(\la y\rtimes_{\mu}\psi \, ,\, z\rtimes_{\mu}1^0 \ra_{D\rtimes_{\beta}H^0}) \\
& =E_1^{\beta}([\psi_{(1)}^* \cdot_{\beta} \la y\, ,\, z \ra_D ]\rtimes_{\beta}\psi_{(2)}^* ) \\
& =\psi(e) \la y \, , \, z \ra_D .
\end{align*}
On the other hand,
$$
\la E_1^{\mu, \Psi}(y\rtimes_{\mu}\psi) \, , \, z\rtimes_{\mu}1^0 \ra_{D\rtimes_{\widehat{\sigma}}H^0 } 
=\Psi^{-1}(\la \psi(e)y \, , \, z \ra_{D\rtimes_{\beta}H^0})=\psi(e) \la y \, , \, z \ra_D .
$$
Hence Condition (6) holds. Therefore, we obtain the conclusion.
\end{proof}

\begin{lemma}\label{lem:jones} With the above notation, for any $y\in Y$,
$$
E_1^{\mu, \Psi}(e_A \cdot y\cdot e_B )=\frac{1}{N}E^X (y) .
$$
\end{lemma}
\begin{proof} By the definition of $E_1^{\mu, \Psi}$ and Lemma \ref{lem:projection},
$$
E_1^{\mu, \Psi}(e_A \cdot y \cdot e_B )=E_1^{\mu}((1\rtimes_{\widehat{\rho}}\tau)\cdot y \cdot \Psi(e_B ))
=E_1^{\mu}((1\rtimes_{\widehat{\rho}}\tau)\cdot y\cdot (1\rtimes_{\beta}\tau)) .
$$
Also,
\begin{align*}
(1\rtimes_{\widehat{\rho}}\tau)\cdot y\cdot (1\rtimes_{\beta}\tau) & =(1\rtimes_{\widehat{\rho}}\tau)
\cdot (y\rtimes_{\mu}1^0 )\cdot (1\rtimes_{\beta}\tau)=(1\rtimes_{\widehat{\rho}}\tau)\cdot (y\rtimes_{\mu}\tau) \\
& =[\tau_{(1)}\cdot_{\mu}y]\rtimes_{\mu}\tau_{(2)}\tau' =[\tau\cdot_{\mu}y]\rtimes_{\mu}\tau ,
\end{align*}
where $\tau' =\tau$. Hence
$$
E_1^{\mu, \Psi}(e_A \cdot y \cdot e_B )=E_1 ^{\mu}([\tau\cdot_{\mu}y]\rtimes_{\mu}\tau )
=[\tau\cdot_{\mu}y]\tau(e)=\frac{1}{N}E^X (y)
$$
by Lemma \ref{lem:action}.
\end{proof}

\begin{prop}\label{prop:iso1} With the above notation, there is a $C_1 -D_1$-equivalence
bimodule isomorphism $\theta$ of $Y_1$ onto $(Y\rtimes_{\mu}H^0 )_{\Psi}$ such that
$E_1^{\mu, \Psi}=E^Y \circ\theta$.
\end{prop}
\begin{proof} This is immediate by Lemma \ref{lem:jones} and \cite [Theorem 6.13]{KT4:morita}.
\end{proof}

Next, modifying the discussions of \cite [Section 5]{KT5:Hopf}, we shall show that
there is a $C^*$-Hopf algebra automorphism $f^0$ of $H^0$ such that
$$
\widehat{\beta}\circ\Psi=(\Psi\otimes f^0 )\circ\widehat{\widehat{\sigma}} ,
$$
where $\widehat{\beta}$ is the dual coaction of $\beta$ and $\widehat{\widehat{\sigma}}$ is the
second dual coaction of $(\sigma, v)$.

\begin{lemma}\label{lem:restriction} With the above notation, $\Psi|_{B' \cap D_1}$, the restriction
of $\Psi$ to $B' \cap D_1$ is an isomorphism of $B' \cap D_1$ onto $B' \cap (D\rtimes_{\beta}H^0)$.
\end{lemma}
\begin{proof} It suffices to show that $\Psi(d)\in B' \cap (D\rtimes_{\beta}H^0 )$ for any $d\in B' \cap D_1$.
For any $d\in B' \cap D_1$, $b\in B$,
$$
\Psi(d)b=\Psi(d)\Psi(b)=\Psi(db)=\Psi(b)\Psi(d)=b\Psi(d) .
$$
Hence $\Psi(d)\in B' \cap(D\rtimes_{\beta}H^0 )$ for any $d\in B' \cap D_1$.
\end{proof}

By Lemma \cite [Lemma 5.8]{KT5:Hopf}, $B' \cap D_1 =1\rtimes_{\sigma, v}1\rtimes_{\widehat{\sigma}}H^0$.
Also, we have the next lemma.

\begin{lemma}\label{lem:commutant} With the above notation, $B' \cap (D\rtimes_{\beta}H^0 )=
1\rtimes_{\sigma, v}1\rtimes_{\beta}H^0$.
\end{lemma}
\begin{proof} We note that $\psi\cdot_{\mu}x=\epsilon^0 (\psi)x$ for any $\psi\in H^0$, $x\in X$
by \cite [Lemma 3.2]{KT5:Hopf}. Thus by the definition of $\beta$, $\psi\cdot_{\beta}(b\rtimes_{\sigma, v}1)
=\epsilon^0 (\psi)(b\rtimes_{\sigma, v}1 )$ for any $\psi\in H^0$, $b\in B$ (See \cite [Sectoin 4]{KT5:Hopf}).
Hence in the same way as in the proof of \cite [Lemma 5.8]{KT5:Hopf}, we obtain the conclusion.
\end{proof}

Since $\Psi(1\rtimes_{\widehat{\sigma}}\tau)=1\rtimes_{\beta}\tau$ by Lemma \ref{lem:projection} and
$\Psi(d)=d$ for any $d\in D$, in the same way as in \cite [Lemma 5.6]{KT5:Hopf}, we can
see that there is an isomorphism $\widehat{\Psi}$ of $D_2$ onto
$D\rtimes_{\beta}H^0 \rtimes_{\widehat{\beta}}H$ satisfying that
$$
\widehat{\Psi}|_{D_1}=\Psi, \quad E_2^{\beta}\circ\widehat{\Psi}=\Psi\circ E_3^{\sigma, v}, \quad
\widehat{\Psi}(1_D \rtimes_{\widehat{\sigma}}1^0 \rtimes_{\widehat{\widehat{\sigma}}}e)
=1_D \rtimes_{\beta}1^0 \rtimes_{\widehat{\beta}}e ,
$$
where $\widehat{\widehat{\sigma}}$ is the second dual coaction of $(\sigma, v)$,
$\widehat{\beta}$ is the dual coaction of $\beta$, $D_2 =D_1 \rtimes_{\widehat{\widehat{\sigma}}}H$ and
$E_3^{\sigma, v}$ and $E_2^{\beta}$ are the canonical conditional expectations from
$D_2$ and $D\rtimes_{\beta}H^0 \rtimes_{\widehat{\beta}}H$ onto $D_1$ and $D\rtimes_{\beta}H^0$,
respectively. Furthermore, in the same way as in the above or \cite [Section 5]{KT5:Hopf},
$\widehat{\Psi}|_{D' \cap D_2}$ is an isomorphism of $D' \cap D_2$ onto
$D' \cap(D\rtimes_{\beta}H^0 \rtimes_{\widehat{\beta}}H)$. Since
$$
B' \cap D_1 =B' \cap (D\rtimes_{\beta}H^0 )
=1\rtimes_{\sigma, v}1\rtimes_{\beta}H^0
$$
by Lemma \ref{lem:commutant}, we identify $B' \cap D_1$ and $B' \cap (D\rtimes_{\beta}H^0 )$
with $H^0$. Let $f^0 =\Psi|_{B' \cap D_1}$ and we regard $f^0$ as a $C^*$-algebra
automorphism of $H^0$. By the proof of \cite [Lemma 5.9]{KT5:Hopf}, we can see that
\begin{align*}
& N^2 (E_2^{\sigma, v}\circ E_3^{\sigma, v})((1\rtimes_{\widehat{\sigma}}
\psi\rtimes_{\widehat{\widehat{\sigma}}}1)(1\rtimes_{\widehat{\sigma}}1^0
\rtimes_{\widehat{\widehat{\sigma}}}e)
(1\rtimes_{\widehat{\sigma}}\tau\rtimes_{\widehat{\widehat{\sigma}}}1)
(1\rtimes_{\widehat{\sigma}}1^0 \rtimes_{\widehat{\widehat{\sigma}}}h))=\psi(e) , \\
& N^2 (E_1^{\beta}\circ E_2^{\beta})((1\rtimes_{\beta}
\psi\rtimes_{\widehat{\beta}}1)(1\rtimes_{\beta}1^0
\rtimes_{\widehat{\beta}}e)
(1\rtimes_{\beta}\tau\rtimes_{\widehat{\beta}}1)
(1\rtimes_{\beta}1^0 \rtimes_{\widehat{\beta}}h))=\psi(e) , 
\end{align*}
for any $h\in H$, $\psi\in H^0$. Hence in the same way as in the proof of
\cite [Lemma 5.9]{KT5:Hopf}, we can see that $f^0$ is a $C^*$-Hopf algebra automorphism of $H^0$.

\begin{lemma}\label{lem:equation} With the above notation, $\widehat{\beta}\circ\Psi
=(\Psi\otimes f^0 )\circ\widehat{\widehat{\sigma}}$.
\end{lemma}
\begin{proof} This can be proved in the same way as in the proof of \cite [Lemma 5.10]{KT5:Hopf}.
\end{proof}

\begin{lemma}\label{lem:equivalent} With the above notation, $\widehat{\beta}(1_D \rtimes_{\beta}\tau)$ is
Murray-von Neumann equivalent to $(1_D \rtimes_{\beta}\tau)\otimes 1^0$ in
$(D\rtimes_{\beta}H^0 )\otimes H^0$.
\end{lemma}
\begin{proof} By Lemmas \ref{lem:projection}, \ref{lem:equation},
$$
\widehat{\beta}(1_D \rtimes_{\beta}\tau)=\widehat{\beta}(\Psi(1_D \rtimes_{\widehat{\sigma}}\tau))
=(\Psi\otimes f^0 )(\widehat{\widehat{\sigma}}(1_D \rtimes_{\widehat{\sigma}}\tau)) .
$$
By \cite [Proposition 3.19]{KT1:inclusion}, $\widehat{\widehat{\sigma}}(1_D\rtimes_{\widehat{\sigma}}\tau)$
is Murray-von Neumann equivalent to $(1\rtimes_{\widehat{\sigma}}\tau)\otimes 1^0$ in
$D_1 \otimes H^0$. Hence we obtain the conclusion by Lemma \ref{lem:projection}.
\end{proof}

\begin{lemma}\label{lem:saturated} With the above notation, $\beta$ is saturated, that is,
the action of $H$ on $D$ induced by $\beta$ is saturated in the sense of 
Szyma\'nski and Peligrad \cite{SP:saturated}.
\end{lemma}
\begin{proof} By the definition of $\widehat{\sigma}$,
$$
\overline{D_1 (1_D \rtimes_{\widehat{\sigma}}\tau)D_1 }=D_1 .
$$
Since $\Psi$ is an isomorphism of $D_1$ onto $D\rtimes_{\beta}H^0 $,
$$
\overline{(D\rtimes_{\beta}H^0 )(1_D\rtimes_{\beta}\tau)(D\rtimes_{\beta}H^0 )}
=\Psi(\overline{D_1 (1_D \rtimes_{\widehat{\sigma}}\tau)D_1 })=\Psi(D_1)=D\rtimes_{\beta}H^0
$$
by Lemma \ref{lem:action}.
Hence $\beta$ is saturated.
\end{proof}

Since $\beta$ is saturated by Lemma \ref{lem:saturated}, there is the conditional expectation
$E^{D^{\beta}}$ from $D$ onto $D^{\beta}$ defined by
$$
E^{D^{\beta}}(d)=\tau\cdot_{\beta} d
$$
for any $d\in D$ (See \cite [Proposition 2.12]{SP:saturated}), where $D^{\beta}$ is the fixed-point
$C^*$-subalgebra of $D$ for $\beta$. Also, since $\widehat{\beta}(1\rtimes_{\beta}\tau)$ is
Murray-von Neumann equivalent to $(1\rtimes_{\beta}\tau)\otimes 1^0$ in
$(D\rtimes_{\beta}H^0 )\otimes H^0$ by Lemma \ref{lem:equivalent}, there is a twisted coaction
$(\gamma, w)$ of $H^0$ on $D^{\beta}$ and an isomorphism $\pi_D$ of $D$ onto
$D^{\beta}\rtimes_{\gamma, w}H$ satisfying
$$
E_1^{\sigma, v}=E^{D^{\beta}}\circ\pi_D, \quad
\psi\cdot_{\widehat{\gamma}}\pi_D (d)=\pi_D (\psi\cdot_{\beta}d)
$$
for any $d\in D$, $\psi\in H^0$ by \cite [Proposition 6.1, 6.4 and Theorem 6.4]{KT1:inclusion}.
We identify $D^{\beta}\rtimes_{\gamma, w}H$ and $E_1^{\gamma, w}$ with $D$ and $E^{D^{\beta}}$
by the above isomorphism $\pi_D$, respectively. We show that $B=D^{\beta}$. By the definition of $\beta$,
$B\subset D^{\beta}$. Let $F$ be the conditional expectation of $D^{\beta}$ onto $B$
defined by $F=E_1^{\sigma, v}|_{D^{\beta}}$, the restriction of $E_1^{\sigma, v}$ to $D^{\beta}$.
Since $E_1^{\sigma, v}$ is of Watatani index-finite type, there is a quasi-basis $\{(d_i , d_i ^* )\}_{i=1}^n$
for $E_1^{\sigma, v}$. Then $F\circ E_1^{\gamma, w}$ is also a conditional expectation from
$D$ onto $B$. Since $B' \cap D=\BC1$, by \cite [Proposition 1.4.1]{Watatani:index},
$$
E_1^{\sigma, v}=F\circ E_1^{\gamma, w} .
$$

\begin{lemma}\label{lem:index} With above notation, $F$ is of Watatani index-finite type
and its Watatani index, $\Ind_W (F)\in \BC1$.
\end{lemma}
\begin{proof} We claim that $\{(E_1^{\gamma, w}(d_i )\, , \, E_1^{\gamma, w}(d_i^* ))\}_{i=1}^n$
is a qusi-basis for $F$. Indeed, for any $d\in D^{\beta}$,
\begin{align*}
\sum_{i=1}^n E_1^{\gamma, w}(d_i )F(E_1^{\gamma, w}(d_i^* )d) & =
\sum_{i=1}^n E_1^{\gamma, w}(d_i )F(E_1^{\gamma, w}(d_i^* )E_1^{\gamma , w}(d)) \\
& =\sum_{i=1}^n E_1^{\gamma, w}(d_i )(F\circ E_1^{\gamma, w})(d_i^* E_1^{\gamma, w}(d)) \\
& =\sum_{i=1}^n E_1^{\gamma, w}(d_i )E_1^{\sigma, v}(d_i^* E_1^{\gamma, w}(d)) \\
& =\sum_{i=1}^n E_1^{\gamma, w}(d_i E_1^{\sigma, v}(d_i^* E_1^{\gamma, w}(d))) \\
& =E_1^{\gamma, w}(E_1^{\gamma, w}(d))=d
\end{align*}
since $E_1^{\sigma, v}=F\circ E_1^{\gamma, w}$ and $E_1^{\gamma, w}(d)=d$ for any $d\in D^{\beta}$.
Hence $F$ is of Watatani index-finite type. Also, $\Ind_W (F)\in (D^{\beta})'\cap D^{\beta} \subset
B' \cap D=\BC 1$ by \cite [Proposition 1.2.8]{Watatani:index}.
\end{proof}

\begin{lemma}\label{lem:fix} With the above notation, $B=D^{\beta}$.
\end{lemma}
\begin{proof} It suffices to show that $\Ind_W (F)=1$. By \cite [Proposition 1.7.1]{Watatani:index},
$$
\Ind_W (E_1^{\sigma, v})=\Ind_W (F)\Ind_W (E_1^{\gamma, w}) .
$$
By \cite [Proposition 3.18]{KT1:inclusion} $\Ind_W (E_1^{\sigma, v})=\Ind_W (E_1^{\gamma, w})=N$.
Hence $\Ind_W (F)=1$. Therefore, we obtain the conclusion by \cite {Watatani:index}.
\end{proof}

Let $Y^{\mu}=\{y\in Y \, | \, \mu(y)=y\otimes 1^0 \}$. By \cite [Theorem 4.9]{Kodaka:equivariance},
there are a twisted coaction $\lambda$ of $H$ on $Y^{\mu}$ and a Hilbert $A\rtimes_{\rho, u}H-
B\rtimes_{\gamma, w}H$-bimodule isomorphism $\pi_Y$ of $Y^{\mu}\rtimes_{\lambda}H$ onto
$Y$ such that
$$
\psi\cdot_{\mu}\pi_Y (x\rtimes_{\lambda}h)=\pi_Y (\psi\cdot_{\widehat{\lambda}}(x\rtimes_{\lambda}h))
$$
for any $x\in Y^{\mu}$, $h\in H$, $\psi\in H^0$. Furthermore, by \cite [Lemma 3.10]{Kodaka:equivariance},
$Y^{\mu}$ is an $A-B$-equivalence bimodule and hence $\pi_Y$ is an $A\rtimes_{\rho, u}H-
B\rtimes_{\gamma, w}H$-equivalence bimodule isomorphism. We identify $Y$ with
$Y^{\mu}\rtimes_{\lambda}H$ by the isomorphism $\pi_Y$. Thus the twisted coactions $(\rho, u)$ and
$(\gamma, w)$ are strongly Morita equivalent with respect to the twisted coaction $\lambda$ of $H$
on the $A-B$-equivalence bimodule $Y^{\mu}$. We show that $Y^{\mu}=X$.

\begin{lemma}\label{lem:fix2} With the above notation, $Y^{\mu}=X$.
\end{lemma}
\begin{proof} By \cite [Lemma 3.2]{KT5:Hopf}, $X\subset Y^{\mu}$. Also, for any $y\in Y^{\mu}$,
$\tau\cdot_{\mu}y=\epsilon^0 (\tau)y=y$. On the other hand, by Lemma \ref {lem:action},
$\tau\cdot_{\mu}y=E^X (y)$. Hence $y=E^X (y)\in X$. Thus we obtain that $Y^{\mu}\subset X$.
\end{proof}

By the above discussions, we obtain the following theorem:

\begin{thm}\label{thm:use} Let $H$ be a finite dimensional $C^*$-Hopf algebra and $H^0$
its dual $C^*$-Hopf algebra. Let $(\rho, u)$ and $(\sigma, v)$ be twisted
coactions of $H^0$ on unital $C^*$-algebras $A$ and $B$,
respectively. Let $A\subset A\rtimes_{\rho, u}H$ and $B\subset B\rtimes_{\sigma, v}H$ be unital
inclusions of unital $C^*$-algebras. We suppose that they are strongly Morita equivalent with respect to
an $A\rtimes_{\rho, u}H-B\rtimes_{\sigma,v}H$-equivalence bimodule $Y$ and its closed subspace $X$.
And we suppose that $A' \cap (A\rtimes_{\rho, u}H)=\BC1$. Then there are a twisted coaction $(\gamma, w)$
of $H^0$ on $B$ and a twisted coaction $\lambda$ of $H^0$ on $X$ satisfying the following:
\newline
$(1)$ $(\rho, u)$ and $(\gamma, w)$ are strongly Morita equivalent with respect to
$\lambda$,
\newline
$(2)$ $B\rtimes_{\sigma, v}H=B\rtimes_{\gamma, w}H$,
\newline
$(3)$ $Y\cong X\rtimes_{\lambda}H$ as $A\rtimes_{\rho, u}H-B\rtimes_{\sigma, v}H$-
equivalence bimodules.
\end{thm}

\section{Image}\label{sec:image} Let $A$ be a unital $C^*$-algebra and $X$ an involutive $A-A$-equivalence
bimodule. Let $A\subset C_X$ be the unital inclusion of unital $C^*$-algebras induced by $X$.
We suppose that $A' \cap C_X =\BC1$. Let $f_A$ be the homomorphism of $\Pic(A, C_X )$ onto
$\Pic(A)$ defined in Preliminaries, that is,
$$
f_A ([M, N])=[M]
$$
for any $(M, N)\in\Equi (A, C_X )$. In this section, we shall compute $\Ima f_A$, the image of $f_A$.
\par
Let $E^A$ be the conditional expectation from $C_X$ onto $A$ defined in Section \ref{sec:kernel} and
let $e_A$ be the Jones projection for $E^A$. Since $E^A$ is of Watatani index-finite type by
\cite [Lemma 3.4]{KT0:involution}, there is the $C^*$-basic construction of the inclusion $A\subset C_X$ for
$E^A$, which is the linking $C^*$-algebra $L_X$ for $X$, that is,
$$
L_X =\{\begin{bmatrix} a & x \\
\widetilde{y} & b \end{bmatrix} \, | \, a, b\in A, \, \, x, y\in X \} .
$$
By \cite [Lemma 2.6]{KT0:involution}, we can see that there is the action $\alpha^X$ of $\BZ_2$,
the group of order two, on $C_X$ defined by
$$
\alpha^X (\begin{bmatrix} a & x \\
\widetilde{x^{\natural}} & a \end{bmatrix})=\begin{bmatrix} a & -x \\
-\widetilde{x^{\natural}} & a \end{bmatrix}
$$
for any $\begin{bmatrix} a & x \\
\widetilde{x^{\natural}} & a \end{bmatrix}\in C_X$ and that
$L_X \cong C_X \rtimes_{\alpha^X}\BZ_2$ as $C^*$-algebras.
We note that we regard an action $\beta$ of $\BZ_2$ on a unital $C^*$-algebra $B$
as the automorphism $\beta$ of $B$ with $\beta^2 =\id$ on $B$. We identify $L_X$ with
$C_X\rtimes_{\alpha^X}\BZ_2$. Let $M$ be an $A-A$-equivalence bimodule satisfying that
$$
\widetilde{M}\otimes_A X\otimes_A M\cong X
$$
as involutive $A-A$-equivalence bimodules. Then by the proof of \cite [Lemma 5.11]{Kodaka:bundle},
we can see that there is an element $(M, C_M )\in\Equi (A, C_X )$, where $C_M$ is a $C_X -C_X$-equivalence
bimodule induced by $M$, which is defined in \cite [Section 5.]{Kodaka:bundle}.
Next, we show that
$$
\widetilde{M}\otimes_A X\otimes_A M\cong X
$$
as involutive $A-A$-equivalence bimodules for any $(M, N)\in \Equi (A, C_X )$.
Let $(M, N)$ be any element in $\Equi(A, C_X )$. Since $A' \cap C_X =\BC1$,
by \cite [Lemma 4.1]{Kodaka:Picard} there is the unique conditional expectation
$E^M$ from $N$ onto $M$ with respect to $E^A$ and $E^A$. Let $N_1$ be the upward
basic construction of $N$ for $E^M$ (See \cite [Definition 6.5]{KT4:morita}). Then by
\cite [Corollary 6.3]{KT4:morita}, the unital inclusion $C_X \subset L_X$ is strongly Morita
equivalent to itself with respect to $N_1$ and its closed subspace $N$. Hence
by Theorem \ref{thm:use}, there are an action $\gamma$ of
$\BZ_2$ on $C_X$ and an action $\lambda$ of $\BZ_2$ on $N$ satisfying the following:
\vskip 0.1cm
\noindent
(1) The actions $\alpha^X$ and $\gamma$ of $\BZ_2$ on $C_X$ are strongly Morita
equivalent with respect to the action $\lambda$ of $\BZ_2$ on $N$,
\vskip 0.1cm
\noindent
(2) $L_X =C_X \rtimes_{\alpha^X}\BZ_2 =C_X \rtimes_{\gamma}\BZ_2$,
\vskip 0.1cm
\noindent
(3) $N_1 \cong N\rtimes_{\lambda}\BZ_2$ as $L_X -L_X$-equivalence bimodules.
\vskip 0.1cm
\noindent
We identify $N_1$ with $N\rtimes_{\lambda}\BZ_2$.
Let $\widehat{\alpha}^X$ be the dual action of $\alpha^X$, which is an action of $\BZ_2$ on $L_X$.
We regard $\widehat{\alpha}^X$ as an automorphism of $L_X$ with $(\widehat{\alpha}^X)^2 =\id$ on $L_X$,
which is defined by
\begin{align*}
& \widehat{\alpha}^X (\begin{bmatrix} a & x \\
\widetilde{x^{\natural}} & a \end{bmatrix})=\begin{bmatrix} a & x \\
\widetilde{x^{\natural}} & a \end{bmatrix} \, \text{for any} \, \begin{bmatrix} a & x \\
\widetilde{x^{\natural}} & a \end{bmatrix}\in C_X , \\
& \widehat{\alpha}^X (\begin{bmatrix} 1 & 0 \\
0 & 0 \end{bmatrix})
=\begin{bmatrix} 0 & 0 \\
0 &1 \end{bmatrix} .
\end{align*}
Let $(L_X )_{\widehat{\alpha}^X}$ be the involutive $L_X -L_X$-equivalence bimodule induced
by $\widehat{\alpha}^X$, that is, $(L_X )_{\widehat{\alpha}^X}=L_X$ as vector spaces over $\BC$ and
the left $L_X$-action and the left $L_X$-valued inner product on $(L_X )_{\widehat{\alpha}^X}$
are defined in the usual way. The right $L_X$-action and the right $L_X$-valued inner product on
$(L_X )_{\widehat{\alpha}^X}$ are defined as follows: For any $a\in L_X$,
$x, y\in (L_X )_{\widehat{\alpha}^X}$,
$$
x\cdot a=x\widehat{\alpha}^X (a), \quad \la x, y \ra_{L_X}=\widehat{\alpha}^X (x^* y) .
$$
Furthermore, we define the involution $\natural$ as follows: For any $x\in (L_X )_{\widehat{\alpha}^X}$,
$$
x^{\natural} =\widehat{\alpha}^X (x)^* . 
$$
Then by easy computations $(L_X )_{\widehat{\alpha}^X}$ is an involutive
$L_X -L_X$-equivalence bimodule. Let $\widehat{\lambda}$ be the dual action of $\lambda$,
which is an action of $\BZ_2$ by linear automorphisms of $N_1 =N\rtimes_{\lambda}Z_2$
such that
\begin{align*}
\widehat{\alpha}^X ({}_{L_X} \la m, n \ra) & ={}_{L_X} \la \widehat{\lambda}(m) \, , \, \widehat{\lambda}(n) \ra ,\\ 
\widehat{\gamma}(\la m, n \ra_{L_X}) & =\la \widehat{\lambda}(m) \, , \, \widehat{\lambda}(n) \ra_{L_X}
\end{align*}
for any $m, n\in N_1$, where we regard the action $\widehat{\lambda}$ as a linear automorphism of $N_1$
with $\widehat{\lambda}^2 =\id$ on $N_1$. We note that
$$
\widehat{\lambda}(x\cdot m)=\widehat{\alpha}^X (x)\cdot \widehat{\lambda}(m) , \quad
\widehat{\lambda}(m\cdot x)=\widehat{\lambda}(m)\cdot \widehat{\gamma}(x)
$$
for any $m\in N_1$, $x\in L_X$. Since $L_X =C_X \rtimes_{\alpha^X}\BZ_2=C_X \rtimes_{\gamma}\BZ_2$
and $N_1 =N\rtimes_{\lambda}\BZ_2$, in the same way as after the proof of Lemma \ref{lem:action}
and in the proof of Lemma
\ref{lem:equation} or by the discussions of \cite [Section 5]{KT5:Hopf}, there is an automorphism
$\kappa$ of $L_X$ satisfying the following:
$$
\widehat{\gamma}\circ\kappa=\kappa\circ\widehat{\alpha}^X , \quad
\kappa|_{C_X}=\id_{C_X} .
$$
Then $\kappa|_{A' \cap L_X}$ is an automorphism of $A' \cap L_X$. And by \cite {KT0.5:character},
\cite {Kodaka:bundle}, $A' \cap L_X \cong \BC^2$. Since $e_A \in A' \cap L_X$,
$\kappa(e_A )=e_A$ or $1-e_A$. If $\kappa(e_A )=e_A$, $\kappa=\id_{L_X}$ since
$\kappa|_{C_X}=\id_{C_X}$. Hence $\widehat{\gamma}=\widehat{\alpha}^X$. If $\kappa(e_A )=1-e_A$,
$\kappa=\widehat{\alpha}^X$ since $\kappa=\widehat{\alpha}^X =\id$ on $C_X$.
Hence $\widehat{\gamma}\circ\widehat{\alpha}^X =\widehat{\alpha}^X \circ\widehat{\alpha}^X =\id_{L_X}$.
Thus $\widehat{\gamma}=(\widehat{\alpha}^X )^{-1}=\widehat{\alpha}^X$. Then, we obtain the
following:

\begin{lemma}\label{lem:iso1} With the above notation,
$$
\widetilde{N_1}\otimes_{L_X}(L_X )_{\widehat{\alpha}^X}\otimes_{L_X}N_1 \cong (L_X )_{\widehat{\alpha}^X}
$$
as $L_X -L_X$-equivalence bimodules.
\end{lemma}
\begin{proof} We note that $N_1 =N\rtimes_{\lambda}\BZ_2$. Let $\pi$ be the linear map
from $\widetilde{N_1}\otimes_{L_X}(L_X )_{\widehat{\alpha}^X}\otimes_{L_X}N_1 $
to $(L_X )_{\widehat{\alpha}^X}$ defined by
$$
\pi(\widetilde{m}\otimes x\otimes n)=\la x^* \cdot m \, , \, \widehat{\lambda}(n) \ra_{L_X}
$$
for any $m, n \in N_1$, $x\in (L_X )_{\widehat{\alpha}^X}$, where we regard
$\la x^* \cdot m \, , \, \widehat{\lambda}(n) \ra_{L_X}$ as an element in $(L_X )_{\widehat{\alpha}^X}$.
We show that $\pi$ is an involutive $L_X -L_X$-equivalence bimodule isomorphism of
$\widetilde{N_1}\otimes_{L_X}(L_X )_{\widehat{\alpha}^X}\otimes_{L_X}N_1 $
onto $(L_X )_{\widehat{\alpha}^X}$. By routine computations, we can see that $\pi$ is well-defined.
Since $L_X \cdot N_1 =N_1$ by Brown, Mingo and Shen \cite [Proposition 1.7]{BMS:quasi}
and $(L_X )_{\widehat{\alpha}^X}$ is full with respect to the right $L_X$-valued inner product,
$\pi$ is surjective. For any $m, n, m_1 , n_1 \in N_1$, $x , x_1 \in L_X$,
\begin{align*}
\la \pi(\widetilde{m}\otimes x\otimes n) \, , \, \pi(\widetilde{m_1}\otimes x_1 \otimes n_1) \ra_{L_X} & =
\la \la x^* \cdot m \, , \, \widehat{\lambda}(n) \ra_{L_X} \, , \, \la x_1^* \cdot m_1 \, , \,
\widehat{\lambda}(n_1 )\ra_{L_X} \ra_{L_X} \\
& =\widehat{\alpha}^X (\la \widehat{\lambda}(n) \, , \, x^* \cdot m \ra_{L_X}\la x_1^* \cdot m_1 \, , \,
\widehat{\lambda}(n_1 ) \ra_{L_X} ) \\
& =\la n \, , \, \widehat{\lambda}(x^* \cdot m)\ra_{L_X} \la\widehat{\lambda}(x_1^* \cdot m_1) \, , \,
n_1 \ra_{L_X} .
\end{align*}
On the other hand,
\begin{align*}
\la \widetilde{m}\otimes x\otimes n \, , \, \widetilde{m_1}\otimes x_1 \otimes n_1 \ra_{L_X} & =
\la n \, , \, \la \widetilde{m}\otimes x \, , \, \widetilde{m_1}\otimes x_1 \ra_{L_X} \cdot n_1 \ra_{L_X} \\
& =\la n \, , \, \la x \, , \, \la \widetilde{m} \, , \, \widetilde{m_1} \ra_{L_X}\cdot x_1 \ra_{L_X} \cdot n_1
\ra_{L_X} \\
& =\la n \, , \, \la x \, , \, {}_{L_X} \la m \, , \, m_1 \ra \cdot x_1 \ra_{L_X}\cdot n_1 \ra_{L_X} \\
& =\la n \, , \, \widehat{\alpha}^X (x^* \, {}_{L_X} \! \la m \, , \, m_1 \ra x_1 )\cdot n_1 \ra_{L_X} \\
& =\la n \, , \, \widehat{\alpha}^X ({}_{L_X} \! \la x^* \cdot m \, , \, x_1^* \cdot m_1 \ra )\cdot n_1 \ra_{L_X} \\
& =\la n \, , \, \widehat{\lambda}(x^* \cdot m)\cdot\la \widehat{\lambda}(x_1^* \cdot m_1 ) \, , \, n_1 \ra_{L_X}
\ra_{L_X} \\
& =\la n \, , \, \widehat{\lambda}(x^* \cdot m) \ra_{L_X}\la \widehat{\lambda}(x_1^* \cdot m_1 ) \, , \,
n_1 \ra_{L_X} .
\end{align*}
Hence $\pi$ preserves the right $L_X$-valued inner products. 
Similarly, we can see that $\pi$ preserves the left $L_X$-valued inner products.
Thus we can obtain that
$\pi$ is an $L_X -L_X$-equivalence bimodule isomorphism by the remark after \cite [Definition 1.1.18]
{JT:KK}. Furthermore,
\begin{align*}
\pi(\widetilde{m}\otimes x\otimes n)^{\natural} & =
\la x^* \cdot m \, , \, \widehat{\lambda}(n) \ra_{L_X}^{\natural}
=\widehat{\alpha}^X (\la x^* \cdot m \, , \, \widehat{\lambda}(n) \ra_{L_X}^* ) \\
& =\widehat{\alpha}^X (\la \widehat{\lambda}(n) \, , \, x^* \cdot m \ra_{L_X} )
=\la n \, , \, \widehat{\lambda}(x^* \cdot m) \ra_{L_X} .
\end{align*}
On the other hand,
\begin{align*}
\pi((\widetilde{m}\otimes x\otimes n)^{\natural}) & =\pi(\widetilde{n}\otimes x^{\natural}\otimes m)
= \la (x^{\natural})^* \cdot n \, , \, \widehat{\lambda}(m) \ra_{L_X} \\
& =\la \widehat{\alpha}^X (x)\cdot n \, , \, \widehat{\lambda}(m) \ra_{L_X}
=\la n \, , \, \widehat{\alpha}^X (x)^* \cdot \widehat{\lambda}(m) \ra_{L_X} \\
& =\la n \, , \, \widehat{\lambda}(x^* \cdot m) \ra_{L_X} .
\end{align*}
Hence $\pi$ preserves involutions $\natural$. Therefore, we obtain the conclusion.
\end{proof}

We regard $e_A L_X$ as an $A-L_X$-equivalence bimodule in the usual way,
where we identify $e_A L_X e_A$ with $A$. Also, we regard $L_X e_A$ as an $L_X -A$-
equivalence bimodule in the usual way. We note that $L_X e_A \cong \widetilde{e_A L_X}$
as $L_X -A$-equivalence bimodules by the map $xe_A \in L_X e_A \mapsto e_A x^*
\in \widetilde{e_A L_X}$. In the same way as in \cite [Section 3]{KT0:involution}, we regard
$e_A L_X (1-e_A )$ as an involutive $A-A$-equivalence bimodule.

\begin{lemma}\label{lem:iso2} With the above notation,
$$
e_A L_X \otimes_{L_X}(L_X )_{\widehat{\alpha}^X} \otimes_{L_X} L_X e_A \cong e_A L_X (1-e_A )\cong X
$$
as involutive $A-A$-equivalence bimodules.
\end{lemma}
\begin{proof} By \cite [Theorem 3.11]{KT0:involution}, we can see that $e_A L_X (1-e_A )\cong X$
as involutive $A-A$-equivalence bimodules. Let $\pi$ be the linear map from $e_A L_X \otimes_{L_X}
(L_X )_{\widehat{\alpha}^X}\otimes_{L_X}L_X e_A$ to $e_A L_X (1-e_A )$ defined by
$$
\pi(e_A x\otimes y\otimes ze_A )=e_A xy\widehat{\alpha}^X (ze_A )
=e_A xy\widehat{\alpha}^X (z)(1-e_A )
$$
for any $x, y, z \in L_X$. We note that $\widehat{\alpha}^X (e_A )=1-e_A$ by
\cite [Remark 2.7]{KT0:involution}. Clearly $\pi$ is surjective. For any $x, y, z, x_1 , y_1 , z_1 \in L_X$,
\begin{align*}
& {}_A \la \pi(e_A x\otimes y\otimes ze_A ) \, ,\, \pi(e_A x_1 \otimes y_1 \otimes z_1 e_A ) \ra \\
& ={}_A \la e_A xy\widehat{\alpha}^X (z)(1-e_A ) \, , \, e_A x_1 y_1 \widehat{\alpha}^X (z_1 )(1-e_A ) \ra \\
& =e_A xy\widehat{\alpha}^X (z)(1-e_A )\widehat{\alpha}^X (z_1^* )y_1^*x_1^* e_A .
\end{align*}
On the other hand,
\begin{align*}
& {}_A \la e_A x\otimes y\otimes ze_A \, , \, e_A x_1 \otimes y_1\otimes z_1 e_A \ra \\
& ={}_A \la e_A x\cdot {}_{L_X} \la y\otimes ze_A \, , \, y_1 \otimes z_1 e_A \ra \, , \, e_A x_1 \ra \\
& =[e_A x\cdot {}_{L_X} \la y\otimes ze_A \, , \, y_1 \otimes z_1 e_A \ra ]x_1^* e_A \\
& =e_A x\, {}_{L_X} \la y\otimes ze_A \, , \, y_1 \otimes z_1 e_A \ra x_1^* e_A \\
& =e_A x\, {}_{L_X} \!\la y\cdot{}_{L_X} \la ze_A \, , \, z_1 e_A \ra \, , \, y_1 \ra x_1^* e_A \\
& =e_A x \, {}_{L_X} \! \la y\cdot ze_A z_1^* \, , \, y_1 \ra x_1^* e_A \\
& =e_A x \, {}_{L_X} \! \la y\widehat{\alpha}^X (ze_A z_1^* ) \, , \, y_1 \ra x_1^* e_A \\
& =e_A xy\widehat{\alpha}^X (ze_A z_1^* )y_1^* x_1^* e_A \\
& =e_A xy\widehat{\alpha}^X (z)(1-e_A )\widehat{\alpha}^X (z_1^* )y_1^* x_1^* e_A .
\end{align*}
Hence $\pi$ preserves the left $A$-valued inner products.
Similarly, we can see that $\pi$ preserves the right $A$-valued inner products.
Thus we can obtain that $\pi$ is an
$A-A$-equivalence bimodule isomorphism by the remark after \cite [Definition 1.1.18]{JT:KK}.
Furthermore,
\begin{align*}
\pi(e_A x\otimes y\otimes ze_A )^{\natural} & =(e_A xy\widehat{\alpha}^X (z)(1-e_A ))^{\natural}
=\widehat{\alpha}^X (1-e_A )z^* \widehat{\alpha}^X (y^* x^* e_A ) \\
& =e_A z^* \widehat{\alpha}^X (y^* x^* )(1-e_A ) .
\end{align*}
On the other hand,
$$
\pi((e_A x\otimes y\otimes ze_A )^{\natural})=\pi(e_A z^* \otimes \widehat{\alpha}^X (y^* )\otimes x^* e_A )
=e_A z^* \widehat{\alpha}^X (y^* x^* )(1-e_A ) .
$$
Hence $\pi$ preserves the involutions $\natural$. Therefore, we obtain the conclusion.
\end{proof}

\begin{lemma}\label{lem:iso3} With the above notation, $e_A L_X \otimes_{L_X}C_X \cong A$
as $A-A$-equivalence bimodules, where $C_X$ is regarded as an $L_X -A$-equivalence bimodule
in the usual way and $A$ is regarded as the trivial $A-A$-equivalence bimodule.
\end{lemma}
\begin{proof} Let $\pi$ be the linear map from $e_A L_X \otimes_{L_X}C_X$ to $A$ defined by
$$
\pi(e_A ae_A b \otimes c)=e_A a e_A b\cdot c =E^A (a)e_A b\cdot c
=E^A (a)E^A (bc)
$$
for any $a, b, c\in C_X$. Clearly $\pi$ is surjective. For any $a, b, c, a_1 , b_1 , c_1 \in C_X$,
\begin{align*}
{}_A \la \pi (e_A ae_A b\otimes c) \, , \, \pi(e_A a_1 e_A b_1 \otimes c_1 ) \ra &=
{}_A \la E^A (a)E^A (bc) \, , \, E^A (a_1 )E^A (b_1 c_1 ) \ra \\
& =E^A (a)E^A (bc)E^A (c_1^* b_1^* )E^A (a_1^* ) .
\end{align*}
On the other hand,
\begin{align*}
{}_A \la e_A a e_A b\otimes c \, , \, e_A a_1 e_A b_1\otimes c_1 \ra & =
{}_A \la e_A ae_A b\cdot {}_{L_X} \la c \, , \, c_1 \ra \, , \, e_A a_1 e_A b_1 \ra \\
& ={}_A \la e_A ae_A b\cdot ce_A c_1^* \, , \, e_A a_1 e_A b_1 \ra \\
& ={}_A \la e_A a e_A bce_A c_1^* \, , \, e_A a_1 e_A b_1 \ra \\
& ={}_A \la e_A E^A (a)E^A (bc)c_1^* \, , \, e_A E^A (a_1 )b_1 \ra \\
& =e_A E^A (a)E^A (bc)c_1^* b_1^* E^A (a_1^* )e_A \\
& =E^A (a)E^A (bc)E^A (c_1^* b_1^* )E^A (a_1^* )e_A .
\end{align*}
Since we identify $A$ with $Ae_A$ by the map $a\in A\mapsto ae_A \in Ae_A$,
$\pi$ preserves the left $A$-valued inner products.
Similarly, we can see that $\pi$ preserves the right $A$-valued inner products.
Thus by the remark
after \cite [Definition 1.1.18]{JT:KK}, we obtain the conclusion.
\end{proof}

\begin{prop}\label{prop:iso4} For any $(M, N)\in\Equi (A, C_X)$,
$$
X\cong \widetilde{M}\otimes_A X\otimes _A M
$$
as involutive $A-A$-equivalence bimodules.
\end{prop}
\begin{proof} By Lemmas \ref{lem:iso2}, \ref{lem:iso1},
\begin{align*}
X & \cong e_A L_X \otimes_{L_X}(L_X)_{\widehat{\alpha}^X}\otimes_{L_X} L_X e_A \\
& \cong e_A L_X \otimes_{L_X}\widetilde{N_1}\otimes_{L_X}(L_X)_{\widehat{\alpha}^X}\otimes_{L_X} N_1
\otimes_{L_X} L_X e_A \\
& \cong e_A L_X \otimes_{L_X}\widetilde{N_1}\otimes_{L_X}L_X e_A \otimes_A X
\otimes_A e_A L_X\otimes_{L_X} N_1 \otimes_{L_X} L_X e_A .
\end{align*}
as involutive $A-A$-equivalence bimodules. Since $N_1 =C_X \otimes_A M\otimes_A \widetilde{C_X}$,
$$
e_A L_X \otimes_{L_X} N_1 \otimes_{L_X} L_X e_A
=e_A L_X \otimes_{L_X}C_X \otimes_A M \otimes_A \widetilde{C_X}\otimes_{L_X}L_X e_A ,
$$
where $C_X$ is regarded as an $L_X -A$-equivalence bimodule. Hence by Lemma \ref{lem:iso3},
\begin{align*}
e_A L_X \otimes_{L_X}N_1 \otimes_{L_X}L_X e_A & \cong
e_A L_X \otimes_{L_X}C_X \otimes_A M \otimes_A [e_A L_X \otimes_{L_X} C_X]^{\widetilde{}} \\
& \cong A\otimes_A M\otimes_A A \cong M
\end{align*}
as $A-A$-equivalence bimodules. Therefore,
\begin{align*}
X & \cong[e_A L_X \otimes_{L_X} N_1 \otimes_{L_X} L_X e_A ]^{\widetilde{}}
\otimes_A X \otimes_A [e_A L_X \otimes_{L_X} N_1 \otimes_{L_X} L_X e_A ] \\
& \cong \widetilde{M}\otimes_A X\otimes_A M
\end{align*}
as involutive $A-A$-equivalence bimodules.
\end{proof}

\begin{thm}\label{thm:image} Let $A$ be a unital $C^*$-algebra and $X$ an involutive
$A-A$-equivalence bimodule. Let $A\subset C_X$ be the unital inclusion of unital
$C^*$-algebras induced by $X$. We suppose that $A' \cap C_X =\BC1$.
Let $f_A$ be the homomorphism of $\Pic(A, C_X )$
to $\Pic(A)$ defined by
$$
f_A ([M, N])=[M]
$$
for any $(M, N)\in\Equi (A, C_X )$. Then the image of $f_A$ is:
\begin{align*}
\Ima f_A =\{[M]\in\Pic(A) \, | \,  & \text{$M$ is an $A-A$-equivalence bimodule with} \\
& \text{$X\cong\widetilde{M}\otimes_A X\otimes_A M$} \\
& \text{as involutive $A-A$-equivalence bimodules}\} .
\end{align*}
\end{thm}
\begin{proof} This is immediate by Proposition \ref{prop:iso4} and the proof of \cite [Lemma 5.11]
{Kodaka:bundle}.
\end{proof}

\section{A homomorphism}\label{sec:homo} In this section, we shall construct a homomorphism
$g$ of $\Ima f_A$ to $\Pic(A, C_X )$ with $f_A \circ g =\id$ on $\Ima f_A$. Let $M$ be an
$A-A$-equivalence bimodule with $X\cong \widetilde{M}\otimes_A X\otimes_A M$ as
involutive $A-A$-equivalence bimodules. Let $\Phi_M$ be an involutive $A-A$-equivalence
bimodule isomorphism of $\widetilde{M}\otimes_A X\otimes_A M$ onto $X$ and let $\widetilde{\Phi_M}$
be the involutive $A-A$-equivalence bimodule isomorphism of
$\widetilde{M}\otimes_A \widetilde{X}\otimes_A M$ onto $\widetilde{X}$ induced by
$\Phi_M$ (See \cite [Section 5]{Kodaka:bundle}). Let $\Psi_M$ and $\widetilde{\Psi_M}$ be the $A-A$-
equivalence bimodule isomorphism of $X\otimes_A M$ onto $M\otimes_A X$ and the $A-A$-equivalence
bimodule isomorphism of $\widetilde{X}\otimes_A M$ onto $M\otimes_A \widetilde{X}$ indced by
$\Phi_M$ and $\widetilde{\Phi_M}$, which are defined in \cite [Section 5]{Kodaka:bundle}, respectively.
Let $C_M$ be the linear span of the set
$$
C_M^X =\{\begin{bmatrix} m_1 & m_2 \otimes x \\
m_2 \otimes \widetilde{x}^{\natural} & m_1 \end{bmatrix} \, | \, m_1 , m_2 \in M, x\in X \} .
$$
Also, let $C_M^1$ be the linear span of the set
$$
{}^X C_M =\{\begin{bmatrix} m_1 & x \otimes m_2 \\
\widetilde{x}^{\natural}\otimes m_2 & m_1 \end{bmatrix} \, | \, m_1 , m_2 \in M, x\in X \} .
$$
As mentioned in \cite [Section 5]{Kodaka:bundle}, we identify $C_M$ with $C_M^1$ by $\Psi_M$ and
$\widetilde{\Psi_M}$. In the same way as in \cite [Section 5]{Kodaka:bundle}, we define the left
$C_X$-action and the right $C_X$-action on $C_M$ as follows:
\begin{align*}
& \begin{bmatrix} a & x \\
\widetilde{x}^{\natural} & a \end{bmatrix}\cdot \begin{bmatrix} m_1 & m_2 \otimes y \\
m_2 \otimes \widetilde{y}^{\natural} & m_1 \end{bmatrix} \\
& =\begin{bmatrix} a\otimes m_1 +x\otimes m_2 \otimes\widetilde{y}^{\natural} & a\otimes m_2 \otimes y
+x\otimes m_1 \\
\widetilde{x}^{\natural} \otimes m_1 +a\otimes m_2 \otimes\widetilde{y}^{\natural} &
\widetilde{x}^{\natural}\otimes m_2 \otimes y +a\otimes m_1 \end{bmatrix} , \\
& \begin{bmatrix} m_1 & m_2 \otimes y \\
m_2 \otimes \widetilde{y}^{\natural} & m_1 \end{bmatrix} \cdot \begin{bmatrix} a & x \\
\widetilde{x}^{\natural} & a \end{bmatrix} \\
& =\begin{bmatrix}m_1 \otimes a+m_2 \otimes y\otimes \widetilde{x}^{\natural} &
m_1 \otimes x+m_2 \otimes y\otimes a \\
m_2 \otimes\widetilde{y}^{\natural} \otimes a +m_1 \otimes\widetilde{x}^{\natural} &
m_2 \otimes\widetilde{y}^{\natural}\otimes x+m_1 \otimes a \end{bmatrix} .
\end{align*}
for any $a\in A$, $m_1, m_2 \in M$, $x, y\in X$. But we identify $A\otimes_A M$, $M\otimes_A A$
and $X\otimes_A \widetilde{X}$, $\widetilde{X}\otimes_A X$ with $M$ and $A$ by the isomorphisms
defined by
\begin{align*}
& a\otimes m\in A\otimes_A M\mapsto a\cdot m\in M , \\
& m\otimes a \in M\otimes_A A\mapsto m\cdot a\in M , \\
& x\otimes \widetilde{y}\in X\otimes_A \widetilde{X}\mapsto {}_A \la x, y \ra\in A , \\
& \widetilde{x}\otimes y\in \widetilde{X}\otimes_A X \mapsto \la x, y \ra_A \in A ,
\end{align*}
respectively and we identify $X\otimes_A M$ and $\widetilde{X}\otimes_A M$ with $M\otimes_A X$ and
$M\otimes_A \widetilde{X}$ by $\Psi_M$ and $\widetilde{\Psi_M}$, respectively.
By the above identifications, the right hand-sides of the above equations are in $C_M$.
Before we define a left $C_X$-valued inner product and a right $C_X$-valued inner product on $C_M$,
we define a conjugate linear map on $C_M$,
$$
\begin{bmatrix} m_1 & m_2 \otimes x \\
m_2 \otimes\widetilde{x^{\natural}} & m_1 \end{bmatrix}\in C_m \mapsto
\begin{bmatrix} m_1 & m_2 \otimes x \\
m_2 \otimes\widetilde{x^{\natural}} & m_1 \end{bmatrix}^{\widetilde{}} \in C_m
$$
by
$$
\begin{bmatrix} m_1 & m_2 \otimes x \\
m_2 \otimes\widetilde{x^{\natural}} & m_1 \end{bmatrix}^{\widetilde{}}
=\begin{bmatrix} \widetilde{m_1} & x^{\natural}\otimes \widetilde{m_2} \\
\widetilde{x}\otimes\widetilde{m_2} & \widetilde{m_1} \end{bmatrix}
$$
for any $m_1 , m_2 \in M$, $x\in X$.
We define the left $C_X$-valued inner product and the right $C_X$-valued inner product as follows:
\begin{align*}
& {}_{C_X} \la \begin{bmatrix} m_1 & m_2 \otimes x \\
m_2 \otimes \widetilde{x}^{\natural} & m_1 \end{bmatrix} \, , \,
\begin{bmatrix} n_1 & n_2 \otimes y \\
n_2 \otimes\widetilde{y}^{\natural} & n_1 \end{bmatrix} \ra \\
& =\begin{bmatrix} m_1 & m_2 \otimes x \\
m_2 \otimes \widetilde{x}^{\natural} & m_1 \end{bmatrix}\cdot
\begin{bmatrix} n_1 & n_2 \otimes y \\
n_2 \otimes\widetilde{y}^{\natural} & n_1 \end{bmatrix}^{\widetilde{}} \\
& =\begin{bmatrix} m_1 \otimes \widetilde{n_1}+m_2 \otimes x\otimes\widetilde{y}\otimes\widetilde{n_2} &
m_1 \otimes y^{\natural}\otimes\widetilde{n_2}+m_2 \otimes x \otimes\widetilde{n_1} \\
m_2 \otimes \widetilde{x}^{\natural}\otimes\widetilde{n_1}+m_1 \otimes\widetilde{y}\otimes\widetilde{n_2} &
m_2 \otimes \widetilde{x}^{\natural}\otimes y^{\natural}\otimes\widetilde{n_2}+m_1 \otimes\widetilde{n_1}
\end{bmatrix} , \\
&  \la \begin{bmatrix} m_1 & m_2 \otimes x \\
m_2 \otimes \widetilde{x}^{\natural} & m_1 \end{bmatrix} \, , \,
\begin{bmatrix} n_1 & n_2 \otimes y \\
n_2 \otimes\widetilde{y}^{\natural} & n_1 \end{bmatrix} \ra_{C_X} \\
& =\begin{bmatrix} m_1 & m_2 \otimes x \\
m_2 \otimes \widetilde{x}^{\natural} & m_1 \end{bmatrix}^{\widetilde{}} \cdot
\begin{bmatrix} n_1 & n_2 \otimes y \\
n_2 \otimes\widetilde{y}^{\natural} & n_1 \end{bmatrix} \\
& =\begin{bmatrix} \widetilde{m_1}\otimes n_1 +x^{\natural}\otimes\widetilde{m_2}\otimes n_2\otimes
\widetilde{y}^{\natural} &
\widetilde{m_1}\otimes n_2 \otimes y+x^{\natural}\otimes\widetilde{m_2}\otimes n_1 \\
\widetilde{x}\otimes \widetilde{m_2} \otimes n_1 +\widetilde{m_1}\otimes n_2 \otimes\widetilde{y}^{\natural} &
\widetilde{x}\otimes \widetilde{m_2} \otimes n_2 \otimes y+\widetilde{m_1}\otimes n_1 \end{bmatrix} ,
\end{align*}
for any $m_1 , m_2 , n_1 , n_2 \in M$, $x, y\in X$, where we regard the tensor product
as a product on $C_M$ in the formal manner. We denote it by $``\cdot "$.
Also, we identify $A\otimes_A M$, $M\otimes_A A$ and
$X\otimes_A \widetilde{X}$, $\widetilde{X}\otimes_A X$ with $M$ and $A$ by the same isomorphisms as
above and we identify $X\otimes_A M$ and $\widetilde{X}\otimes_A M$ with $M\otimes_A X$ and
$M\otimes_A \widetilde{X}$ by $\Psi_M$ and $\widetilde{\Psi_M}$. By the above identifications, we can
define the left $C_X$-valued and the right $C_X$-valued inner products. In the same way as above,
we can define the left $C_X$-action and the right $C_X$-valued action on $C_M^1$ and the left
$C_X$-valued inner product and the right $C_X$-valued inner product on $C_M^1$. Since we identify
$C_M$ with $C_M^1$ by $\Psi_M$ and $\widetilde{\Psi_M}$, we can see that $C_M$ and $C_M^1$
are $C_X - C_X$-equivalence bimodules by \cite [Lemma 5.10]{Kodaka:bundle} and that each of them
agrees with the other by routine computations (See \cite [Section 5]{Kodaka:bundle}.
We identify $C_M$ with $C_M^1$ as $C_X -C_X$-equivalence bimodules by the isomorphisms
$\Psi_M$ and $\widetilde{\Psi_M}$ and we denote them by
the same symbol $C_M$. Furthermore, by \cite [Lemma 5.11]{Kodaka:bundle},
$(M, C_M)\in\Equi(A, C_C )$.
\par
Let $\Phi_M '$ be another involutive $A-A$-equivalence bimodule isomorphism of
$\widetilde{M}\otimes_A X\otimes_A M$ onto $X$ and let $\widetilde{\Phi_M}'$ be the involutive
$A-A$-equivalence bimodule isomorphism of $\widetilde{M}\otimes_A \widetilde{X}\otimes_A M$ onto
$\widetilde{X}$ induced by $\Phi_M '$. Let $\Psi_M '$ be the $A-A$-equivalence bimodule isomorphism of
$X\otimes_A M$ onto $M\otimes_A X$ induced by $\Phi_M '$ and let $\widetilde{\Psi_M}'$ be the
$A-A$-equivalence bimodule isomorphism of $\widetilde{X}\otimes_A M$ onto $M\otimes_A \widetilde{X}$
induced by $\widetilde{\Phi_M}'$. Then we can identify $C_M$ with $C_M^1$ by the isomorphisms
$ \Psi_M '$ and $\widetilde{\Psi_M '}$. Hence we can obtain an element in $\Equi (M, C_X )$ by
the above identification. We denote the element by $(M, C_M ' )$.

\begin{lemma}\label{lem:def1} With the above notation, $[M, C_M ]=[M, C_M ' ]$ in $\Pic(A, C_X )$.
\end{lemma}
\begin{proof} We can construct a $C_X -C_X$-equivalence bimodule isomorphism using the $A-A$-
equivalence isomorphisms $\Psi_M$, $\widetilde{\Psi_M}$, $\Psi_M '$, $\widetilde{\Psi_M}'$.
Hence $C_M$ and $C_M '$ are isomorphic as $C_X -C_X$-equivalence bimodules by the
$C_X -C_X$-equivalence bimodule isomorphism, which leaves the diagonal elements in $C_M$ and
$C_M '$ invariant. Thus $[M, C_M ]=[M, C_M ' ]$ in $\Pic(A, C_X )$.
\end{proof}

Let $M_1$ be another $A-A$-equivalence bimodule with $\widetilde{M_1}\otimes_A X\otimes_A M_1
\cong X$ as involutive $A-A$-equivalence bimodules. Let $[M_1 , C_{M_1}]$ be the element in
$\Pic (A, C_X )$ induced by $M_1$ in the above.

\begin{lemma}\label{lem:def2} With the above notation, we suppose that $M$ and $M_1$ are isomorphic
as $A-A$-equivalence bimodules. Then $[M, C_M ]=[M_1 , C_{M_1}]$ in $\Pic(A, C_X )$.
\end{lemma}
\begin{proof}
Since $M\cong M_1$ as $A-A$-equivalence bimodules, there is an $A-A$-equivalence bimodule
isomorphism $\pi$ of $M_1$ onto $M$. Let $\Phi_M$ be an involutive $A-A$-equivalence
bimodules isomorphism of $\widetilde{M}\otimes_A X\otimes_A M$ onto $X$. Then
$\Phi_M \circ(\widetilde{\pi}\otimes \id_X \otimes \pi)$ is an involutive $A-A$-equivalence bimodule
isomorphism of $\widetilde{M_1}\otimes_A X\otimes_A M_1$ onto $X$, where $\widetilde{\pi}$ is the
$A-A$-equivalence bimodule isomorphism of $\widetilde{M_1}$ onto $M$ defined by
$$
\widetilde{\pi}(\widetilde{m})=\widetilde{\pi(m)}
$$
for any $m\in M$. Let $[M, C_M ]$ and $[M_1 , C_{M_1}]$ be the element in $\Pic(A, C_X )$ induced by
$\Phi_M$ and $\Phi_M \circ(\widetilde{\pi}\otimes\id_X \otimes\pi)$. Let $(M, C_M )$ be the element in
$\Equi (A, C_X )$ obtained by using the isomorphism $\Phi_M $ and let $(M_1 , C_{M_1})$ be the element
in $\Equi(A, C_X )$ obtained by using the isomorphism $\Phi_M \circ (\widetilde{\pi}\otimes\id\otimes\pi)$.
Then by the definitions of $(M, C_M )$, $(M_1 , C_{M_1})$ and Lemma \ref{lem:def1}, we obtain that
$[M, C_M ]=[M_1 , C_{M_1}]$ in $\Pic(A, C_X )$. 
\end{proof}

Let $g$ be the map from $\Ima f_A$ to $\Pic(A, C_X )$ defined by
$$
g([M])=[M, C_M ]
$$
for any $[M]\in \Ima f_A$. By Lemmas \ref{lem:def1} and \ref{lem:def2}, $g$ is well-defined.
\par
Let $M$ and $K$ be $A-A$-equivalence bimodules with $\widetilde{M}\otimes_A X\otimes_A M
\cong X$ and $\widetilde{K}\otimes_A X\otimes_A K\cong X$ as
$A-A$-equivalence bimodules, respectively. Let $(M, C_M )$ and $(K, C_K)$ be the elements
in $\Equi (A, C_X )$ induced by $M$ and $K$, respectively. Also, let $(M\otimes_A K, C_{M\otimes_A K})$
be the element in $\Equi (A, C_X )$ induced by $M\otimes_A K$.

\begin{lemma}\label{lem:def3} With the above notation, $C_M \otimes_{C_X}C_K \cong C_{M\otimes_A K}$
as $C_X -C_X$-equivalence bimodules.
\end{lemma}
\begin{proof} Let $\pi$ be the linear map from $C_M \otimes_{C_X}C_K$ onto $C_{M\otimes_A K}$
defined by
\begin{align*}
& \pi(\begin{bmatrix} m_1 & m_2 \otimes x \\
m_2 \otimes\widetilde{x}^{\natural} & m_1 \end{bmatrix}\otimes
\begin{bmatrix} k_1 & y \otimes k_2 \\
\widetilde{y}^{\natural}\otimes k_2 & k_1 \end{bmatrix}) \\
& =\begin{bmatrix} m_1 & m_2 \otimes x \\
m_2 \otimes\widetilde{x}^{\natural} & m_1 \end{bmatrix}\cdot
\begin{bmatrix} k_1 & y \otimes k_2 \\
\widetilde{y}^{\natural}\otimes k_2 & k_1 \end{bmatrix} \\
& =\begin{bmatrix} m_1 \otimes k_1 +m_2 \otimes x\otimes \widetilde{y}^{\natural}\otimes k_2 &
m_1 \otimes y\otimes k_2 +m_2 \otimes x\otimes k_1 \\
m_2 \otimes\widetilde{x}^{\natural}\otimes k_1 +m_1 \otimes\widetilde{y}^{\natural}\otimes k_2 &
m_2 \otimes\widetilde{x}^{\natural}\otimes y\otimes k_2 +m_1 \otimes k_1 \end{bmatrix} ,
\end{align*}
for any $\begin{bmatrix} m_1 & m_2 \otimes x \\
m_2 \otimes\widetilde{x}^{\natural} & m_1 \end{bmatrix}\in C_M^X$ and
$\begin{bmatrix} k_1 & y \otimes k_2 \\
\widetilde{y}^{\natural}\otimes k_2 & k_1 \end{bmatrix}\in{}^X C_K$,
where we regard the tensor product as a product on $C_M$ in the formal manner. But
we identify $A\otimes_A K$ and $X\otimes_A \widetilde{X}$, $\widetilde{X}\otimes_A X$ with $K$ and $A$
by the isomorphisms defined by
\begin{align*}
& a\otimes k\in A\otimes_A K \mapsto a\cdot k \in K , \\
& x\otimes \widetilde{y}\in X\otimes_A \widetilde{X}\mapsto {}_A \la x, y \ra\in A , \\
& \widetilde{x}\otimes y\in \widetilde{X}\otimes_A X\mapsto \la x, y \ra_A \in A,
\end{align*}
respectively. Furthermore, we identify $X\otimes_A K$ and $\widetilde{X}\otimes_A K$ with
$K\otimes_A X$ and $K\otimes_A \widetilde{X}$ as $A-A$-equivalence bimodules by
$\Psi_K$ and $\widetilde{\Psi_K}$, which are defined as above, respectively. Thus,
\begin{align*}
& \pi(\begin{bmatrix} m_1 & m_2 \otimes x \\
m_2 \otimes\widetilde{x}^{\natural} & m_1 \end{bmatrix}\otimes
\begin{bmatrix} k_1 & y \otimes k_2 \\
\widetilde{y}^{\natural}\otimes k_2 & k_1 \end{bmatrix}) \\
& =\begin{bmatrix} m_1 \otimes k_1 +m_2 \otimes {}_A \la x, y^{\natural} \ra \cdot k_2 &
m_1 \otimes \Psi_K (y\otimes k_2 )+m_2 \otimes\Psi_K (x\otimes k_1 ) \\
m_2 \otimes \widetilde{\Psi_K}(\widetilde{x}^{\natural}\otimes k_1 )
+m_1 \otimes\widetilde{\Psi_K}(\widetilde{y}^{\natural}\otimes k_2 ) &
m_2 \otimes \la x^{\natural}, y \ra_A \cdot k_2 +m_1 \otimes k_1 \end{bmatrix} .
\end{align*}
Then by routine computations,
\begin{align*}
& {}_A \la x, y^{\natural} \ra=\la x^{\natural}, y \ra_A , \\
& \widetilde{\Psi_K}(\widetilde{x}^{\natural}\otimes k_1 ) =\sum_{i=1}^n u_i \otimes
\Phi_K (\widetilde{u_i}\otimes x\otimes k_1 )^{\widetilde{}\natural} ,
\end{align*}
\begin{align*}
& \Psi_K (x\otimes k_1 )=\sum_{i=1}^n u_i \otimes \Phi_K (\widetilde{u_i}\otimes x\otimes k_1 ) , \\
& \widetilde{\Psi_K}(\widetilde{y}^{\natural}\otimes k_2 ) =\sum_{i=1}^n u_i \otimes \Phi_K (\widetilde{u_i}
\otimes y\otimes k_2 )^{\widetilde{}\natural} , \\
& \Psi_K (y\otimes k_2 ) =\sum_{i=1}^n u_i \otimes \Phi_K (\widetilde{u_i}\otimes y\otimes k_2 ) ,
\end{align*}
where $\{u_i \}_{i=1}^n$ is a finite subset of $K$ with $\sum_{i=1}^n {}_A \la u_i , u_i \ra =1$ and
$\Phi_K$ and $\widetilde{\Phi_K}$ are as defined in the above. Hence $\pi$ is a linear map from
$C_M \otimes_{C_X}C_K$ to $C_{M\otimes_A K}$. Next, we show that $\pi$ is surjective. We take
elements
$$
\begin{bmatrix} m_1 & m_2 \otimes x \\
m_2 \otimes\widetilde{x}^{\natural} & m_1 \end{bmatrix}\in C_M^X \,, \quad
\begin{bmatrix} k_1 & 0 \\ 0 & k_1 \end{bmatrix}\in {}^X \!C_K .
$$
Then
\begin{align*}
& \pi(\begin{bmatrix} m_1 & m_2 \otimes x \\
m_2 \otimes\widetilde{x}^{\natural} & m_1 \end{bmatrix}\otimes
\begin{bmatrix} k_1 & 0 \\ 0 & k_1 \end{bmatrix}) \\
& =\begin{bmatrix} m_1 \otimes k_1 & m_2 \otimes\Psi_K (x\otimes k_1 ) \\
m_2 \otimes\widetilde{\Psi_K}(\widetilde{x}^{\natural}\otimes k_1 ) & m_1 \otimes k_1
\end{bmatrix} .
\end{align*}
We also take elements
$$
\begin{bmatrix} 0 & m_2 \otimes y \\
m_2 \otimes \widetilde{y}^{\natural} & 0 \end{bmatrix}\in C_M^X \,, \quad
\begin{bmatrix} k_2 & 0 \\
0 & k_2 \end{bmatrix}\in {}^X \! C_K .
$$
Then
\begin{align*}
& \pi(\begin{bmatrix} 0 & m_2 \otimes y \\
m_2 \otimes \widetilde{y}^{\natural} & 0 \end{bmatrix}\otimes
\begin{bmatrix} k_2 & 0 \\
0 & k_2 \end{bmatrix}) \\
& =\begin{bmatrix} 0 & m_2 \otimes\Psi_K (y\otimes k_2) \\
m_2 \otimes\widetilde{\Psi_K}(\widetilde{y}^{\natural}\otimes k_2) & 0 \end{bmatrix} .
\end{align*}
Thus
\begin{align*}
& \pi(\begin{bmatrix} m_1 & m_2 \otimes x \\
m_2 \otimes\widetilde{x}^{\natural} & m_1 \end{bmatrix} \otimes
\begin{bmatrix} k_1 & 0 \\
0 & k_1 \end{bmatrix}
+ \begin{bmatrix} 0 & m_2 \otimes y \\
m_2 \otimes\widetilde{y}^{\natural} & 0 \end{bmatrix} \otimes
\begin{bmatrix} k_2 & 0 \\
0 & k_2 \end{bmatrix} ) \\
& =\begin{bmatrix} m_1 \otimes k_1 & m_2\otimes\Psi_K (x\otimes k_1 +y\otimes k_2 ) \\
m_2 \otimes\widetilde{\Psi_K}(\widetilde{x}^{\natural}\otimes k_1 +\widetilde{y}^{\natural}\otimes k_2 ) &
m_1 \otimes k_1 \end{bmatrix} .
\end{align*}
Since $\Psi_K$ and $\widetilde{\Psi_K}$ are isomorphisms of $X\otimes_A K$ and
$\widetilde{X}\otimes_A K$ onto $K\otimes_A X$ and $K\otimes_A \widetilde{X}$, respectively,
we can see that $\pi$ is surjective. Furthermore, by the definitions of $\pi$ and the left and the right
$A$-valued inner products on $C_M$, $C_K$ and $C_{M\otimes K}$, we can easily see that $\pi$
preserves the left and the right $A$-valued inner products.
Indeed, let ${\mathcal M}, {\mathcal M_1}\in C_M$ and ${\mathcal K}, {\mathcal K_1}\in C_K$.
Then
\begin{align*}
{}_{C_X} \la \pi({\mathcal M}\otimes{\mathcal K}) \, , \, \pi({\mathcal M_1}\otimes{\mathcal K_1}) \ra
& ={}_{C_X} \la {\mathcal M}\cdot{\mathcal K} \, , \, {\mathcal M_1}\cdot{\mathcal K_1} \ra \\
& ={\mathcal M}\cdot{\mathcal K}\cdot({\mathcal M_1}\cdot{\mathcal K_1})^{\widetilde{}}
={\mathcal M}\cdot{\mathcal K}\cdot\widetilde{{\mathcal K_1}}\cdot\widetilde{{\mathcal M_1}} .
\end{align*}
Also,
\begin{align*}
{}_{C_X} \la {\mathcal M}\otimes{\mathcal K} \, , \, {\mathcal M_1}\otimes{\mathcal K_1} \ra
& ={}_{C_X} \la {\mathcal M} \cdot{}_{C_X} \la {\mathcal K} \, ,\, {\mathcal K_1} \ra \, , \, {\mathcal M_1} \ra \\
& ={}_{C_A} \la {\mathcal M}\cdot{\mathcal K}\cdot\widetilde{{\mathcal K_1}} \, , \, {\mathcal M_1} \ra
={\mathcal M}\cdot{\mathcal K}\cdot\widetilde{{\mathcal K_1}}\cdot\widetilde{{\mathcal M_1}} .
\end{align*}
Hence $\pi$ preserves the left $C_X$-valued inner products. Similarly, we can see that $\pi$
preserves the right $C_X$-valued inner products. Therefore, $\pi$ is a $C_X - C_X$-equivalence
bimodule isomorphism of $C_M \otimes_{C_X}C_K$ onto $C_{M\otimes_A K}$ by the remark after
\cite [Definition 1.1.18]{JT:KK}.
\end{proof}

\begin{prop}\label{prop:homo1} With the above notation, $g$ is a homomorphism of $\Ima f_A$
to $\Pic(A, C_X )$ with $f_A \circ g=\id$ on $\Ima f_A$.
\end{prop}
\begin{proof} This is immediate by Lemma \ref{lem:def3} and the definition of $g$.
\end{proof}

We give the main result of this paper.

\begin{thm}\label{thm:homo2} Let $A$ be a unital $C^*$-algebra and $X$ an involutive $A-A$-equivalence
bimodule. Let $A\subset C_X$ be the unital inclusion of unital $C^*$-algebras induced by
$X$. We suppose that $A' \cap C_X =\BC1$. Let $f_A$ be the homomorphism of $\Pic (A, C_X )$
to $\Pic(A)$ defined by
$$
f_A ([M, N])=[M]
$$
for any $(M, N)\in \Equi(A, C_X )$. Then $\Pic(A, C_X )$ is isomorphic to a semi-direct product
group of $\BT$ by the group
\begin{align*}
\{[M]\in\Pic(A) \, | \,  & \text{$M$ is an $A-A$-equivalence bimodule with }
\text{$X\cong\widetilde{M}\otimes_A X\otimes_A M$} \\
& \text{as involutive $A-A$-equivalence bimodules}\} .
\end{align*}
\end{thm}
\begin{proof} This is immediate by Proposition \ref{prop:ker3}, Theorem \ref{thm:image} and
Proposition \ref{prop:homo1}.
\end{proof}


\begin{thebibliography}{99}






\bibitem{BGR:linking}L. G. Brown, P. Green and M. A. Rieffel,
{\it Stable isomorphism and strong Morita equivalence of $C^*$-algebras},
Pacific J. Math.,
{\bf 71}
(1977),
349--363.

\bibitem{BMS:quasi}L. G. Brown, J. Mingo and N-T. Shen,
{\it Quasi-multipliers and embeddings of Hilbert $C^*$-bimodules},
Can. J. Math.,
{\bf 46}
(1994),
1150--1174.





\bibitem{JT:KK}K. K. Jensen and K. Thomsen,
{\it Elements of KK-theory},
Birkh$\ddot a$user,
1991.





\bibitem{Kodaka:equivariance} K. Kodaka,
{\it Equivariant Picard groups of $C^*$-algebras with finite dimensional $C^*$-Hopf algebra coactions},
Rocky Mountain J. Math.,
{\bf 47}
(2017),
1565--1615.

\bibitem{Kodaka:Picard} K. Kodaka,
{\it The Picard groups for unital inclusions of unital $C^*$-algebras},
Acta Sci. Math. (Szeged),
\bf
86
\rm
(2020), 183-207.

\bibitem{Kodaka:bundle} K. Kodaka,
{\it Equivalence bundles over a finite group and the strong Morita equivalence for unital inclusions of unital
$C^*$-algebras},
preprint, arXiv:1905.10001v1.

\bibitem{KT0:involution}K. Kodaka and T. Teruya,
{\it Involutive equivalence bimodules and inclusion of $C^*$-algebras with Watatani index 2},
J. Operator Theory,
{\bf 57}
(2007),
3--18.

\bibitem{KT0.5:character}K. Kodaka and T. Teruya,
{\it A characterization of saturated $C^*$-algebraic bundles over finite groups},
J. Aust. Math. Soc.,
{\bf 88}
(2010),
363--383.

\bibitem{KT1:inclusion}K. Kodaka and T. Teruya,
{\it Inclusions of unital $C^*$-algebras of index-finite type with depth 2 induced by saturated
actions of finite dimensional $C^*$-Hopf algebras},
Math. Scand.,
{\bf 104}
(2009),
221--248.


\bibitem{KT3:equivalence}K. Kodaka and T. Teruya,
{\it The strong Morita equivalence for coactions of a finite dimensional $C^*$-Hopf algebra on
unital $C^*$-algebras},
Studia Math.,
{\bf 228}
(2015),
259--294.

\bibitem{KT4:morita}K. Kodaka and T. Teruya,
{\it The strong Morita equivalence for inclusions of $C^*$-algebras and conditional expectations for
equivalence bimodules}, J. Aust. Math. Soc.,
{\bf 105}
(2018),
103--144.

\bibitem{KT5:Hopf}K. Kodaka and T. Teruya,
{\it Coactions of a finite dimensional $C^*$-Hopf algebra on unital $C^*$-algebras, unital inclusions of
unital $C^*$-algebras and the strong Morita equivalence}, preprint, arXiv:1706.09530, Studia Math., to appear.






\bibitem{Rieffel:rotation}M. A. Rieffel,
{\it $C^*$-algebras associated with irrational rotations},
Pacific J. Math.,
{\bf 93}
(1981),
415--429.

\bibitem{Sweedler:Hopf}
M. E. Sweedler,
{\it Hopf algebras},
Benjamin, New York, 1969.

\bibitem{Szymanski:subfactor}W. Szyma\'nski,
{\it Finite index subfactors and Hopf algebra crossed products},
Proc. Amer. Math. Soc.,
{\bf 120}
(1994),
519--528.

\bibitem{SP:saturated}W. Szyma\'nski and C. Peligrad,
{\it Saturated actions of finite dimensional Hopf {\rm *}-algebras on  
$C^*$-algebras},
Math. Scand.,
{\bf 75}
(1994),
217--239.


\bibitem{Watatani:index}Y. Watatani,
{\it Index for $C^*$-subalgebras},
Mem. Amer. Math. Soc.,
{\bf 424}, Amer. Math. Soc., 1990.


\end{thebibliography}
\end{document}